\documentclass[review,onefignum,onetabnum]{siamart171218}

\usepackage{lipsum}
\usepackage{mathrsfs}
\usepackage{amsfonts}
\usepackage{graphicx}
\usepackage{epstopdf}
\usepackage{algorithmic}
\usepackage{amssymb}
\usepackage{amsmath}

\ifpdf
  \DeclareGraphicsExtensions{.eps,.pdf,.png,.jpg}
\else
  \DeclareGraphicsExtensions{.eps}
\fi

\newsiamremark{hypothesis}{Hypothesis}
\crefname{hypothesis}{Hypothesis}{Hypotheses}
\newsiamthm{claim}{Claim}
\newtheorem{remark}{Remark}
\nolinenumbers

\headers{Nield-Kuznetsov functions}{T. M. Dunster}

\title{Nield-Kuznetsov functions and Laplace transforms of parabolic cylinder functions}

\author{T. M. Dunster\thanks{Department of Mathematics and Statistics, San Diego State University, 5500 Campanile Drive, San Diego, CA 92182-7720, USA. 
  (\email{mdunster@sdsu.edu}, \url{https://tmdunster.sdsu.edu}).}}

\usepackage{amsopn}

\makeatletter
\newcommand*{\addFileDependency}[1]{
  \typeout{(#1)}
  \@addtofilelist{#1}
  \IfFileExists{#1}{}{\typeout{No file #1.}}
}
\makeatother

\nolinenumbers

\ifpdf
\hypersetup{
  pdftitle={Nield-Kuznetsov functions and Laplace transforms of parabolic cylinder functions},
  pdfauthor={T. M. Dunster}
}

\begin{document}

\maketitle

\begin{abstract}
  Nield-Kuznetsov functions of the first kind are studied, which are solutions of an inhomogeneous Weber parabolic cylinder differential equation, and have applications in fluid flow problems. Connection formulas are constructed between them, numerically satisfactory solutions of the homogeneous version of the differential equation, and a new complementary Nield-Kuznetsov function. Asymptotic expansions are then derived that are uniformly valid for large values of the parameter and unbounded real and complex values of the argument. Laplace transforms of the parabolic cylinder functions $W(a,x)$ and $U(a,x)$ are subsequently shown to be explicitly represented in terms of the complementary Nield-Kuznetsov function and closely related functions, and from these uniform asymptotic expansions are derived for these integrals.
\end{abstract}

\begin{keywords}
  {Parabolic cylinder functions, WKB methods, Asymptotic expansions, Laplace transforms}
\end{keywords}

\begin{AMS}
  33C10, 34E20, 34E05, 44A10 
\end{AMS}

\section{Introduction} 
\label{sec1}

The parametric Nield-Kuznetsov function of the first kind is given by \cite{Zaytoon:2016:WID}, \cite{Nield:2009:TET}

\begin{equation} \label{01}
N_{W}(a,x)=
W(a,x) \int_{0}^x W(a,-t)dt 
-W(a,-x)\int_{0}^x W(a,t)dt,
\end{equation}
which is a solution of 
\begin{equation} \label{02}
\frac{d^{2}y}{dx^{2}}+\left(\frac{1}{4}x^{2}-a \right)y=-1.
\end{equation}
In (\ref{01}) the Weber parabolic cylinder functions $W(a,\pm x)$ are numerically satisfactory solutions of the homogeneous form of (\ref{02}) when $a$ and $x$ are real, and for definitions and properties see \cite[Sect. 12.14]{NIST:DLMF}. Throughout this paper we assume that $a$ is real, but $x$ can be real or complex.

The function $N_{W}(a,x)$ appears in the study of fluid flow through a variable permeability porous layer, which involve Brinkman equations. These equations appear in the study of fast-moving fluids in porous media in which the kinetic potential from fluid velocity, pressure, and gravity drives the flow. In the study of flow through certain porous layers the Brinkman equations can be reduced to inhomogeneous forms of various linear second order differential equations. When the distribution of the reciprocal of the permeability is linear we get an inhomogeneous form of Airy’s equation \cite{Nield:2009:TET}. We mention that this equation was studied in \cite{Dunster:2020:ASI} where the inhomogeneous term was either a polynomial or an exponential.

For a quadratic distribution of the reciprocal of the permeability we get the parabolic cylinder equation, one form of which is (\ref{02}) studied here; see \cite{Alzahrani:2016:NKF} and \cite{Hamdan:2017:IBV}. In \cite{Zaytoon:2016:WID} series solutions for $N_{W}(a,x)$ for small $x$ were derived, and various other representations were derived in the above references.

In \cite{Zaytoon:2016:WID} the authors state "The current work involves real arguments of the Weber functions. In the general analysis of Weber equation and associated functions with complex arguments, the introduction of a Nield-Kuznetsov type function is still challenging and requires further consideration." In this paper we address this open problem, as described below.

From its definition the function $N_{W}(a,x)$ is readily seen to be even (and hence $N_{W}'(a,0)=0$), and also vanishes at $x=0$. These two properties define $N_{W}(a,x)$ uniquely, since if $y_{p}(a,x)$ is another even particular solution of (\ref{02}), then for some constant $c(a)$ it must be related by
\begin{equation} \label{02aa}
y_{p}(a,x)=
N_{W}(a,x) +c(a)\left\{W(a,x) + W(a,-x)\right\}.
\end{equation}
But if also $y_{p}(a,0)=0$ then $c(a)W(a,0)=0$ which implies $c(a)=0$ since $W(a,0) \neq 0$ (see (\ref{15}) below). Alternatively, uniqueness follows from the existence and uniqueness theorem for solutions of linear second order differential equations with given initial conditions at an ordinary point.

We introduce and study the complementary parametric Nield-Kuznetsov function
\begin{equation} \label{05}
\hat{N}_{W}(a,x)
 =W(a,x)\int_{x}^{\infty} W(a,-t)dt
-W(a,-x) \int_{x}^{\infty} W(a,t)dt 
\end{equation}
the negative of which is also a particular solution of (\ref{02}). This has the advantage over $N_{W}(a,x)$ that it is uniquely defined by its recessive behavior at $x=\infty$, and hence $y(a,x)=-N_{W}(a,x)$ is a numerically satisfactory particular solution of (\ref{02}). Specifically, as we shall show, for large positive $x$ this function decays monotonically and is $\mathcal{O}(x^{-2})$, whereas all other solutions of (\ref{02}) (including $N_{W}(a,x)$) are oscillatory with an amplitude that decays more slowly, namely asymptotically proportional to $x^{-1/2}$.

Furthermore, we shall extend the definition of $N_{W}(a,x)$ and $\hat{N}_{W}(a,x)$ to complex argument $x=z$ (say). We shall demonstrate that $\hat{N}_{W}(a,z) \rightarrow 0$ as $z \rightarrow \infty$ for $|\arg(z)| \leq \frac{1}{2}\pi - \delta$ ($\delta>0$), and also as $a \rightarrow \infty$ uniformly in part of the right half-plane that contains $\{z: \, \Re(z)\geq 2\sqrt{a} \}$. In contrast, all other solutions of (\ref{02}) are exponentially large as $z \rightarrow \infty$, and also as $a \rightarrow \infty$, in at least three of the four quadrants.

When $x$ is real and small $N_{W}(a,x)$ is useful for solving (\ref{02}) for the initial value problem $y(a,0)=\alpha$, $y'(a,0)=\beta$; see \cite{Hamdan:2017:IBV}. However, a solution involving $\hat{N}_{W}(a,x)$ is more practicable computationally. This is because as $x$ grows $N_{W}(a,x)$ becomes almost indistinguishable from a homogeneous solution. In addition, as we shall show it grows exponentially as $a \rightarrow \infty$ for all nonzero argument. 

The solution to the initial value problem in question is readily found to be
\begin{equation} \label{05h}
y(a,x)=
C^{+}(a)W(a,x)+C^{-}(a)W(a,-x)+N_{W}(a,x),
\end{equation}
where
\begin{equation} \label{05k}
\hat{C}^{\pm}(a)=\mp \beta W(a,0)-\alpha W'(a,0)
\end{equation}
In deriving this we used the relation $W(a,0)W'(a,0)=-\frac{1}{2}$, which follows from the Wronskian of $W(a,\pm x)$ \cite[Eq. 12.14.3]{NIST:DLMF}, or alternatively from their values at $x=0$ \cite[Eqs. 12.14.1, 12.14.2]{NIST:DLMF}
\begin{equation} \label{15}
W\left(a,0\right)=2^{-3/4}\left|\frac{\Gamma\left(\tfrac{1}{4}+\tfrac{1}{2}ia\right)}{\Gamma\left(\tfrac{3}{4}+\tfrac{1}{2}ia\right)}\right|^{1/2},
\end{equation}
and 
\begin{equation} \label{16}
W'\left(a,0\right)=-2^{-1/4}\left|\frac{\Gamma\left(\tfrac{3}{4}+
\tfrac{1}{2}ia\right)}{\Gamma\left(\tfrac{1}{4}+\tfrac{1}{2}ia\right)}\right|^{1/2}.
\end{equation}

We note that $W(a,0)=\mathcal{O}(a^{-1/4})$ and $W'(a,0)=\mathcal{O}(a^{1/4})$ as $a \rightarrow \infty$ (see (\ref{22a}) and (\ref{22b}) below). Now (\ref{15}) and (\ref{16}) of course can also be used in (\ref{05k}), and then (\ref{05h}) can be used for small values of $x$. For all other values of $x$ the Nield-Kuznetsov function $N_{W}(a,x)$ in (\ref{05h}) should be replaced by an expression involving the complementary function given by (\ref{05}). This can be done using an appropriate connection formula, which we derive in \cref{sec2} (see (\ref{06}) and \cref{thm2.7main} below). In fact, $\hat{N}_{W}(a,x)$ can be used for all $x$; see (\ref{22w}) along with \cref{lem2.4,lem2.6} for the case $x$ small.

For large real or complex values of $z$ and/or large positive $a$ we present uniform asymptotic expansions for $\hat{N}_{W}(a,z)$ and related functions. These are given in \cref{sec3}, employing recent results of \cite{Dunster:2020:ASI} and \cite{Dunster:2021:UAP}. They involve Scorer functions, which are solutions of the inhomogeneous Airy equation, and slowly varying coefficient functions.

In \cref{sec4,sec5} we study the Laplace transforms of parabolic cylinder functions, which turn out to be related to the Nield-Kuznetsov functions. Laplace transforms for a wide range of other special functions are well documented. For example, see \cite[Eqs. 9.10.14 and 10.22.49, Sect. 13.10(ii)]{NIST:DLMF} for Airy, Bessel and confluent hypergeometric functions, respectively, and \cite[pp. 170-176, 179–181]{Erdelyi:1954:TI1} for orthogonal polynomials and associated Legendre functions, respectively. However for the parabolic cylinder functions there do not appear to be comparable results in the literature. Existing formulas either involve an exponential factor with a quadratic exponent, and/or the argument of the parabolic cylinder function having a square root. 

Consider the standard parabolic cylinder function $U(a,z)$, which is a solution of the equation
\begin{equation}  \label{UEQ}
\frac{d^{2}y}{dz^{2}}-\left(\frac{1}{4}z^{2}+a \right)y=0,
\end{equation}
having the integral representation \cite[Eqs. 12.5.1 and 12.5.4]{NIST:DLMF}
\begin{equation} \label{UINT}
U(a,z)=\frac{e^{-\frac{1}{4}z^{2}}}
{\Gamma\left(\tfrac{1}{2}+a \right)}
\int_0^\infty {t^{a-\frac{1}{2}}
e^{-\frac{1}{2}t^{2}-zt}dt}
\quad \left(\Re(a)>-\tfrac{1}{2} \right).
\end{equation}
Its fundamental property is that it is the unique solution of (\ref{UEQ}) that is exponentially small as $z \rightarrow \infty$ for $|\arg(z)| \leq \frac{1}{4}\pi - \delta$ ($\delta>0$). Specifically, as $z \rightarrow \infty$
\begin{equation} \label{64a}
U(a,z)\sim z^{-a-\frac{1}{2}}
e^{-\frac{1}{4}z^{2}}
\quad \left(\vert \arg(z)\vert \le \tfrac{3}{4}\pi -\delta
<\tfrac{3}{4}\pi\right).
\end{equation}

Then, as an example of a known Laplace transform involving this function but also an exponential with quadratic exponent, from \cite[Eq. 4.20.1]{Erdelyi:1954:TI1} we find that for $\Re(\lambda)>0$ and $\Re(a)>-\frac{1}{2}$
\begin{equation} \label{05m}
\int_{0}^{\infty }e^{-\lambda t}
e^{\frac{1}{4}t^{2}} U(a,t) dt
=\frac{1}{\Gamma\left(\tfrac{1}{2}+a\right)}
\int_{0}^{\infty}
\frac{x^{a-{\frac{1}{2}}}
e^{-\frac{1}{2}x^{2}}}{\lambda+x}dx.
\end{equation}

In \cref{sec4,sec5} we fill the gap by obtaining explicit representations of the Laplace transforms of $W(a,\pm t)$ and $U(a,t)$, respectively, without any extraneous multiplicative terms. We find explicit expressions in terms of the complementary Nield-Kuznetsov function $\hat{N}_{W}(a,z)$ and related functions. From these and the asymptotics of \cref{sec3} we are then able to obtain powerful asymptotic expansions for the Laplace transforms, as $a \rightarrow \infty$, that are uniformly valid for all real and complex values of the Laplace variable $\lambda$. We also include asymptotic expansions for $a$ bounded and $|\lambda|$ large.

We mention that an explicit representation for the Laplace transform for $W(a,\pm x)$ for the case $\lambda=0$ is also provided by (\ref{22eee}) and \cref{thm2.7main} below, and likewise for $U(a,x)$ by (\ref{07}) (with $R=0$ in that formula). Apart from (\ref{07}) all these results appear to be new.

Parabolic cylinder functions are difficult to compute, especially for complex argument, on account of that they also depend on a parameter. This difficulty is compounded in the evaluation of their Laplace transforms, especially when the parameter $a$ is large. The importance of the new expansions of this paper is that they are simple to compute and achieve a uniform reduction of free variables from two to one, since they involve either elementary functions, or a special function of only one variable. In the latter case the approximants in question are the Airy function $\mathrm{Ai}(z)$, and the Scorer function $\mathrm{Hi}(z)$ which is a solution of an inhomogeneous Airy equation; see \cite[Sect. 9.12]{NIST:DLMF} for details. For numerical methods in the evaluation of these functions see \cite{Gil:2000:ONI}, \cite{Gil:2002:F7R} and \cite{GIL:2002:CCA}. Our uniform asymptotic expansions also clearly demonstrate the regions in the complex $\lambda$ plane where the Laplace transforms are exponentially large, exponentially small, and oscillatory.

The general asymptotic expansions used from \cite{Dunster:2020:ASI} and \cite{Dunster:2021:UAP} in \cref{sec3,sec4,sec5} are rigorously established via explicit error bounds, but for brevity we omit details of the bounds in this paper. To avoid phrase repetition, throughout this paper we denote by $\delta$ an arbitrary small positive constant.

\section{Fundamental solutions and connection formulas} 
\label{sec2}

For $z \in \mathbb{C}$ the homogeneous Weber differential equation
\begin{equation} \label{02a}
\frac{d^{2}y}{dz^{2}}+\left(\frac{1}{4}z^{2}-a \right)y=0,
\end{equation}
has numerically satisfactory solutions
\begin{equation} \label{02b}
W_{j}(a,z)
=U\left((-1)^{j}ia,(-i)^{j}ze^{-\pi i/4}\right)
\quad (j=0,1,2,3).
\end{equation}
These solutions are important because each $W_{j}(a,z)$ is recessive (exponentially small) at $z=e^{\pi i/4}i^{j}\infty$, whereas all other linearly independent solutions are exponentially large at that singularity. Note also that $W_{2}(a,z)=W_{0}(a,-z)$ and $W_{1}(a,z)=W_{3}(a,-z)$, and so we can primarily focus on $W_{0}(a,z)$ and $W_{3}(a,z)$.

As $z \rightarrow \infty$ for $j=0$ and $j=3$ we have from (\ref{64a}) and (\ref{02b})
\begin{equation} \label{02c}
W_{j}(a,z)\sim 
e^{-\frac{1}{4}\pi a \pm \frac{1}{8}\pi i}
z^{\mp ia-\frac{1}{2}}
e^{\pm \frac{1}{4}iz^{2}}
\quad \left(
\left\vert \arg\left( ze^{\mp \pi i/4}
\right) \right\vert
\leq \tfrac{3}{4}\pi-\delta \right),
\end{equation}
where the upper signs are taken for $j=0$ and lower signs for $j=3$. From this it is clear that these two form a numerically satisfactory pair of solutions of (\ref{02a}) in the right half-plane. The same is true for $W_{1}(a,z)$ and $W_{2}(a,z)$ in the left half-plane.

With these definitions, the Weber function $W(a,x)$ can then be extended to complex argument $x=z$ via the relation \cite[Eq. 12.14.4]{NIST:DLMF}
\begin{equation} 
\label{22e}
W(a,z)=\sqrt{\tfrac{1}{2}k} \, e^{\pi a/4}\left\{ e^{i\rho
}W_{0}(a,z)+e^{-i\rho }W_{3}(a,z) \right\},
\end{equation}
where
\begin{equation} 
\label{22f}
k=\sqrt {1+e^{2\pi a}} -e^{\pi a}
=\frac{1}{\sqrt {1+e^{2\pi a}}+e^{\pi a}},
\end{equation}
\begin{equation} \label{22h}
\rho = \tfrac{1}{2}\phi_{2}+\tfrac{1}{8}\pi,
\end{equation}
and
\begin{equation} 
\label{22i}
\phi_{2}=\arg \left\{\Gamma\left(\tfrac{1}
{2}+ia\right) \right\}.
\end{equation}
In (\ref{22i}) the branch of $\arg$ is taken to be zero when $a=0$ and then defined by continuity for $a>0$.

From (\ref{02b}) and \cite[Eq. 12.2.19]{NIST:DLMF} we note the connection formulas
\begin{equation} 
\label{71}
W_{0}(a,-z)=-ie^{\pi a}W_{0}(a,z) +
\frac{\sqrt{2\pi}e^{\frac{1}{2}\pi a+\frac{1}{4}\pi i}}{ \Gamma\left( \frac{1}{2}+ia \right)}W_{3}(a,z),
\end{equation}
\begin{equation} 
\label{72}
W_{3}(a,-z)=
\frac{\sqrt{2\pi}e^{\frac{1}{2}\pi a-\frac{1}{4}\pi i}}{ \Gamma\left( \frac{1}{2}-ia \right)}
W_{0}(a,z) +ie^{\pi a}W_{3}(a,z),
\end{equation}
and hence from (\ref{22e}) - (\ref{72})
\begin{equation} \label{22j}
W(a,-z)=-i\sqrt{\frac{1}{2k}} \, e^{\pi a/4}\left\{ e^{i\rho
}W_{0}(a,z)-e^{-i\rho }W_{3}(a,z) \right\}.
\end{equation}

In \cite{Dunster:2021:UAP} solutions were defined and studied for the following inhomogeneous version of (\ref{02a})
\begin{equation} \label{02d}
\frac{d^{2}y}{dz^{2}}+\left(\frac{1}{4}z^{2}-a \right)y=z^{R}
\quad (R=0,1,2,\ldots).
\end{equation}
From Eq. (4.31) of this reference we have the solution that interests us the most, namely
\begin{multline} \label{03}
 W_{R}^{(0,3)}(a,z)
 =ie^{\pi a/2}\left[
W_{3}(a,z)
 \int_{e^{\pi i/4}\infty}^z \right.
 t^{R}W_{0}(a,t)dt \\
\left. -W_{0}(a,z)
 \int_{e^{-\pi i/4}\infty}^z 
 t^{R}W_{3}(a,t)dt
 \right].
\end{multline}
This is defined and converges for all nonnegative integers $R$, and is significant because it is the unique solution that does not grow exponentially in the right half-plane. In particular it is $\mathcal{O}(z^{R-2})$ as $z \rightarrow \infty$ for $|\arg(z)| \leq \frac{1}{2} \pi - \delta$, uniformly for all real values of $a$. We shall primarily focus on the cases $R=0$ and $R=1$.

Firstly, for $R=0$ the following identification allows us to extend $\hat{N}_{W}(a,x)$ to complex argument $z$, and later extract its asymptotic behavior for large parameter and unbounded arguments. 

\begin{theorem}
For the analytic continuation of (\ref{05}) via (\ref{22e}) we have
\begin{multline} 
\label{04}
\hat{N}_{W}(a,z)=W_{0}^{(0,3)}(a,z)
 =ie^{\pi a/2}\left[W_{3}(a,z)
 \int_{e^{\pi i/4}\infty}^z W_{0}(a,t)dt
 \right. \\ \left.
 -W_{0}(a,z) \int_{e^{-\pi i/4}\infty}^z 
 W_{3}(a,t)dt \right].
\end{multline}
\end{theorem}

\begin{proof}
As $x \rightarrow \infty$ it follows from \cite[Eqs. 12.14.17 and 12.14.18]{NIST:DLMF} that
\begin{equation} \label{05a}
W(a,x) \sim \sqrt{\frac{2k}{x}}\cos(\omega),
\end{equation}
and
\begin{equation} \label{05b}
W(a,-x) \sim \sqrt{\frac{2}{kx}}\sin(\omega),
\end{equation}
where $k$ is given by (\ref{22f}),
\begin{equation} \label{05c}
\omega=\tfrac{1}{4}x^{2}-a\ln x+\tfrac{1}{4}\pi+\tfrac{1}{2}\phi_{2},
\end{equation}
and $\phi_{2}$ is given by (\ref{22i}). Now from (\ref{05a}), (\ref{05b}), and (\ref{05c}) one finds that
\begin{equation} \label{05d}
\int_{x}^{\infty} W(a,t)dt \sim
-\sqrt{k}\left(\frac{2}{x}\right)^{3/2}\sin(\omega),
\end{equation}
and
\begin{equation} \label{05e}
\int_{x}^{\infty} W(a,-t)dt \sim
\frac{1}{\sqrt{k}}\left(\frac{2}{x}\right)^{3/2}\cos(\omega).
\end{equation}
So from (\ref{05}), (\ref{05a}), (\ref{05b}), (\ref{05d}), and (\ref{05e}) we deduce that
\begin{equation} \label{05f}
\hat{N}_{W}(a,x) \sim 4/x^2.
\end{equation}

Now any solution $y(a,x)$ of (\ref{02}) can be expressed as
\begin{equation} \label{05g}
y(a,x) = c_{1}(a)W(a,x)+c_{2}(a)W(a,-x)
-\hat{N}_{W}(a,x),
\end{equation}
for some constants $c_{1}(a)$ and $c_{2}(a)$. But from (\ref{05a}), (\ref{05b}), (\ref{05f}), and (\ref{05g}) it is clear that $y(a,x)\sim -4x^{-2}$ as $x \rightarrow \infty$ iff $c_{1}(a)=c_{2}(a)=0$. Therefore $\hat{N}_{W}(a,z)$ and $W_{0}^{(0,3)}(a,z)$ must be equal, since the latter shares the unique behavior (\ref{05f}) as $z=x \rightarrow \infty$, and the negative of both are solutions of (\ref{02}).
\end{proof}

Having defined $\hat{N}_{W}(a,z)$ for complex $z$ we note that
\begin{equation} \label{22t}
\hat{N}_{W}(a,-z)
=\hat{N}_{W}(a,z)-2^{3/4}\pi i e^{\pi a}
\left\{
\frac{2^{ia/2}}{\Gamma(\tfrac{3}{4}-\tfrac{1}{2}ia)}W_{0}(a,z)
-\frac{2^{-ia/2}}{ \Gamma(\tfrac{3}{4}+\tfrac{1}{2}ia)}W_{3}(a,z)
\right\},
\end{equation}
which can be deduced from \cite[Eq. (5.33)]{Dunster:2021:UAP}.

Our primary focus now is the important connection formula, given as follows.
\begin{lemma} \label{lem2.2}
The Nield-Kuznetsov function (\ref{01}) and the complementary version (\ref{05}) are related by
\begin{equation} \label{06}
N_{W}(a,x)=
c_{0}^{+}(a)W(a,x)+c_{0}^{-}(a)W(a,-x)-\hat{N}_{W}(a,x),
\end{equation}
where
\begin{equation} \label{17}
c_{0}^{\pm}(a)=
\mp \hat{N}_{W}'(a,0)W(a,0) - \hat{N}_{W}(a,0)W'(a,0).
\end{equation}
\end{lemma}

\begin{proof}
The function $y(a,x)=-\hat{N}_{W}(a,x)$ is a solution of (\ref{02}) satisfying the initial conditions  $y(a,0)=\alpha=-\hat{N}_{W}(a,0)$ and $y'(a,0)=\beta=-\hat{N}_{W}'(a,0)$. Using these values in (\ref{05h}) and (\ref{05k}) yields the stated result.
\end{proof}

Our aim now is to find explicit and readily computable representations for these coefficients, specifically formulas for $\hat{N}_{W}(a,0)$ and $\hat{N}_{W}'(a,0)$, bearing in mind we already have (\ref{15}) and (\ref{16}). The main result is given by \cref{thm2.7main} below. To arrive at this, we begin by proving some preliminary results. First we have the following.

\begin{lemma} \label{lem2.3}
\begin{equation} \label{12}
 \hat{N}_{W}(a,0)
 =-\pi e^{\pi a/2} \Im\left\{ 
 \frac{e^{-\pi i/4}}{\Gamma\left(\frac{3}{4}
 +\frac{1}{2}ia\right)}
\mathbf{F}\left(\tfrac{1}{2}, 1;
\tfrac{5}{4}-\tfrac{1}{2}ia;\tfrac{1}{2}
\right) \right\},
\end{equation}
where $\mathbf{F}$ is Olver's scaled hypergeometric function \cite[Eq. 15.2.2]{NIST:DLMF}
\begin{equation} \label{08}
\mathbf{F}(a,b;c,z)
=\frac{F(a,b;c,z)}{\Gamma(c)}
=\sum_{s=0}^{\infty}
\frac{(a)_{s}(b)_{s}z^{s}}{\Gamma(c+s)s!}.
\end{equation}
\end{lemma}

\begin{proof}
From (\ref{04}) we have
\begin{multline} \label{09a}
 \hat{N}_{W}(a,0)
 =ie^{\pi a/2}\left[
 W_{0}(a,0) \int_{0}^{e^{-\pi i/4}\infty} W_{3}(a,t)dt
 \right. \\ 
 \left.
- W_{3}(a,0) \int_{0}^{{e^{\pi i/4}\infty}} W_{0}(a,t)dt
 \right].
\end{multline}
The definition of the $W$ functions in terms of $U$ given by (\ref{02b}) then allows us to recast this in the form
\begin{multline} \label{09}
 \hat{N}_{W}(a,0)
 =ie^{\pi a/2}\left[
 U(ia,0) \int_{0}^{e^{-\pi i/4}\infty}
 U\left(-ia,te^{\pi i/4}\right)dt
 \right. \\ \left.
 - U(-ia,0)
 \int_{0}^{e^{\pi i/4}\infty}
 U\left(ia,te^{-\pi i/4}\right)dt
 \right].
\end{multline}
We now make change of integration variables $t \rightarrow t e^{\pm \pi i/4}$ to arrive at
\begin{multline} \label{10a}
 \hat{N}_{W}(a,0)
 =ie^{\pi a/2}\left[ 
 e^{-\pi i/4}U(ia,0)
 \int_{0}^{\infty} U(-ia,t)dt
 \right. \\ \left.
 -e^{\pi i/4}U(-ia,0)
 \int_{0}^{\infty} U(ia,t)dt
 \right],
\end{multline}
which yields
\begin{equation} \label{10}
 \hat{N}_{W}(a,0)
 =-2e^{\pi a/2} \Im\left\{ 
 e^{-\pi i/4}U(ia,0)
 \int_{0}^{\infty} U(-ia,t)dt
 \right\}.
\end{equation}

Next we record the very useful but little-known identity \cite[Eq. 2.11.2.1]{Prudnikov:1986:IAS}
\begin{multline} 
\label{07}
\int_{0}^{\infty}t^{R}U(a,t)dt=
2^{-\frac{1}{2}a-\frac{1}{2}R-\frac{3}{4}}
\sqrt{\pi}R!
\\ \times
\mathbf{F}\left(
\tfrac{1}{2}R+\tfrac{1}{2}, \tfrac{1}{2}R+1;
\tfrac{1}{2}a+\tfrac{1}{2}R+\tfrac{5}{4};\tfrac{1}{2}
\right) \quad (R=0,1,2,\ldots),
\end{multline}
(see also \cite[Sect. 8.1.5]{Magnus:1966:FTS} for a more general form), along with \cite[Eq. 12.2.6]{NIST:DLMF}
\begin{equation} 
\label{11}
U(ia,0)=\frac{\sqrt{\pi}}{2^{\frac{1}{2}ia+\frac{1}{4}}
\Gamma\left(\frac{3}{4}+\frac{1}{2}ia\right)}.
\end{equation}
Setting $R=0$ and replacing $a$ by $-ia$ in (\ref{07}), and putting this along with (\ref{11}) into (\ref{10}), then establishes (\ref{12}).
\end{proof}

Although we now have an explicit representation for $\hat{N}_{W}(a,0)$, the form (\ref{12}) is hard to compute when $a$ is large, since the imaginary part we need to extract is exponentially small relative to the real part of the function in question. For example if $a=20$, then to 4 decimal places
\begin{equation} \label{11b}
 \frac{e^{-\pi i/4}}{\Gamma\left(\frac{3}{4}
 +\frac{1}{2}ia\right)}
\mathbf{F}\left(\tfrac{1}{2}, 1;
\tfrac{5}{4}-\tfrac{1}{2}ia;\tfrac{1}{2}\right)
=7.0167 \times 10^{11} - 0.1262 \, i.
\end{equation}

We overcome this with the following numerically stable representation.
\begin{lemma} \label{lem2.4}
\begin{equation} \label{12cc}
\hat{N}_{W}(a,0)=2e^{\pi a} X(a) - 2Y(a),
\end{equation}
where
\begin{equation} \label{X1}
X(a)=\frac {\left| \Gamma\left( \tfrac{1}{4}+\tfrac{1}{2}ia \right) \right|^{2}}{4\sqrt {2\pi} },
\end{equation}
and
\begin{equation} \label{Y}
Y(a)= \Im\left\{ (1-2ia)^{-1}
F\left(\tfrac{1}{2}, 1;
\tfrac{5}{4}-\tfrac{1}{2}ia;\tfrac{1}{2}\right)
 \right\}.
\end{equation}
\end{lemma}
\begin{proof}
From (\ref{08}) and well-known functional relations for the gamma function $\Gamma(z+1)=z\Gamma(z)$ and
\begin{equation} \label{reflect}
\frac{1}{\Gamma(z)}=
\frac{\sin(\pi z)\Gamma(1-z)}{\pi},
\end{equation}
we obtain
\begin{multline} \label{12bb}
 \frac{e^{-\pi i/4}}{\Gamma\left(\frac{3}{4}
 +\frac{1}{2}ia\right)}
\mathbf{F}\left(\tfrac{1}{2}, 1;
\tfrac{5}{4}-\tfrac{1}{2}ia;\tfrac{1}{2}
\right) \\
 =\frac{e^{-\pi i/4}}{\Gamma\left(\frac{3}{4}
 +\frac{1}{2}ia\right)
 \Gamma\left(\frac{5}{4}
 -\frac{1}{2}ia\right)}
F\left(\tfrac{1}{2}, 1;
\tfrac{5}{4}-\tfrac{1}{2}ia;\tfrac{1}{2}
\right) \\
=\frac {2\left(e^{-\pi a/2}-i e^{\pi a/2}\right)}
{(1-2ia)\pi }
 F\left(\tfrac{1}{2}, 1;
\tfrac{5}{4}-\tfrac{1}{2}ia;\tfrac{1}{2}\right).
\end{multline}
Thus from (\ref{12}) we get (\ref{12cc}) where
\begin{equation} \label{X}
X(a)= \Re\left\{ (1-2ia)^{-1}
F\left(\tfrac{1}{2}, 1;
\tfrac{5}{4}-\tfrac{1}{2}ia;\tfrac{1}{2}\right)
 \right\},
\end{equation}
and $Y(a)$ is given by (\ref{Y}).

Finally, from \cite[Eqs. 5.4.6, 5.5.1, 15.4.6, 15.10.21]{NIST:DLMF} we obtain the relation
\begin{equation} \label{12ff}
F\left(\tfrac{1}{2},1;\tfrac{5}{4}-\tfrac{1}{2}ia;\tfrac{1}{2} \right)=-\frac {1-2ia}{1+2ia}
F\left(\tfrac{1}{2},1;\tfrac{5}{4}+\tfrac{1}{2}ia;\tfrac{1}{2} \right)
+\frac {(1-2ia)\left| \Gamma\left( \tfrac{1}{4}+\tfrac{1}{2}ia \right) \right|^{2}}{2\sqrt {2\pi} }.
\end{equation}
Now divide both sides by $1-2ia$, equate the real parts of both sides, then refer to (\ref{X}), and as a result (\ref{X1}) is verified.
\end{proof}

We shall frequently use Stirling's formula \cite[Eq. 5.11.3]{NIST:DLMF}
\begin{equation} \label{stirling}
\Gamma\left(z\right)\sim e^{-z}z^{z}\left(\frac{2\pi}{z}
\right)^{1/2}\sum_{k=0}^{\infty}\frac{g_{k}}{z^{k}},
\end{equation}
as $z \rightarrow \infty$ for $|\arg(z)| \leq \pi - \delta$, where $g_{0}=1$, $g_{1}=\tfrac{1}{12}$ and subsequent coefficients given by \cite[Eqs. 5.11.5	and 5.11.6]{NIST:DLMF}. From this we have the following.

\begin{lemma} \label{lem2.5} As $a \rightarrow \infty$
\begin{equation} \label{12ee}
\hat{N}_{W}(a,0)
=\sqrt{\frac{\pi}{a}}e^{\pi a/2}
\left\{1+ \mathcal{O}\left(\frac{1}{a^2}\right) \right\}.
\end{equation}
\end{lemma}

\begin{proof}
From (\ref{X1}) and (\ref{stirling}) one finds that
\begin{equation} \label{X2}
X(a)=\frac{\sqrt{\pi} e^{-\pi a/2}}{2\sqrt{a}}
\left\{1+ \mathcal{O}\left(\frac{1}{a^2}\right)
\right\}.
\end{equation}
Also directly from (\ref{08}) and (\ref{Y}) it is seen that
\begin{equation} \label{Y1}
Y(a) \sim \sum_{k=0}^{\infty}
\frac{y_{2k+1}}{a^{2k+1}},
\end{equation}
where the first four coefficients are $y_1=\frac{1}{2}$, $y_3=\frac{1}{4}$, $y_5=\frac{7}{8}$ and $y_7=\frac{139}{16}$.
Thus from (\ref{12cc}) we get (\ref{12ee})
\end{proof}

\begin{remark} A more accurate approximation comes from retaining the gamma function in the leading term, so that
\begin{equation} \label{12eee}
\hat{N}_{W}(a,0)= \frac {e^{\pi a}
\left| \Gamma\left( \tfrac{1}{4}+\tfrac{1}{2}ia \right) \right|^{2}}{2\sqrt {2\pi}  }
\left\{1+ \mathcal{O}\left(\frac{e^{-\pi a/2}}{\sqrt{a}}\right)
\right\} \quad (a \rightarrow \infty).
\end{equation}
\end{remark}

Next we establish a simple explicit value for the derivative at the origin.
\begin{lemma} \label{lem2.6}
\begin{equation} \label{14b}
\hat{N}_{W}'(a,0)
=-\sqrt {\pi} e^{\pi a/2}.
\end{equation}
\end{lemma}

\begin{proof}
From differentiating (\ref{03}) we obtain similarly to (\ref{09})
\begin{multline} \label{13aa}
\hat{N}_{W}'(a,0)
=ie^{\pi a/2}\left[
e^{-\pi i/4} U'(ia,0)
\int_{0}^{e^{-\pi i/4}\infty}
U\left(-ia,te^{\pi i/4}\right)dt
\right. \\ \left.
- e^{\pi i/4} U'(-ia,0)
\int_{0}^{e^{\pi i/4}\infty}
U\left(ia,te^{-\pi i/4}\right)dt
\right],
\end{multline}
and hence following the same steps leading to (\ref{10}) yields
\begin{equation} \label{13}
\hat{N}_{W}'(a,0)
=2e^{\pi a/2}\Re\left\{ 
U'(ia,0)\int_{0}^{\infty}U(-ia,t)dt
\right\}.
\end{equation}
Next from \cite[Eq. 12.2.7]{NIST:DLMF}
\begin{equation} \label{11a}
U'(ia,0)=-\frac{\sqrt{\pi}}{2^{\frac{1}{2}ia-\frac{1}{4}}\Gamma\left(\frac{1}{4}+\frac{1}{2}ia\right)},
\end{equation}
and so from (\ref{07}), (\ref{13}), and (\ref{11a}) we have
\begin{equation} \label{14}
 \hat{N}_{W}'(a,0)
 =-\sqrt{2}\pi e^{\pi a/2} \Re\left\{ 
 \frac{1}{\Gamma\left(\frac{1}{4}+\frac{1}{2}ia\right)}
\mathbf{F}\left(\tfrac{1}{2}, 1;
\tfrac{5}{4}-\tfrac{1}{2}ia;\tfrac{1}{2}
\right)
 \right\}.
\end{equation}
Now from (\ref{08}), (\ref{Y}), and (\ref{X}) and the gamma function recurrence relation we find that
\begin{multline} \label{14a}
 \frac{1}{\Gamma\left(\frac{1}{4}+\frac{1}{2}ia\right)}
\mathbf{F}\left(\tfrac{1}{2}, 1;
\tfrac{5}{4}-\tfrac{1}{2}ia;\tfrac{1}{2}
\right)    \\
=\frac{1}{\Gamma\left(\frac{1}{4}+\frac{1}{2}ia\right)
\Gamma\left(\frac{5}{4}-\frac{1}{2}ia\right)}
F\left(\tfrac{1}{2}, 1;
\tfrac{5}{4}-\tfrac{1}{2}ia;\tfrac{1}{2}
\right)  \\
=\frac{4}{\Gamma\left(\frac{1}{4}+\frac{1}{2}ia\right)
\Gamma\left(\frac{1}{4}-\frac{1}{2}ia\right)}
\frac{F\left(\tfrac{1}{2}, 1;
\tfrac{5}{4}-\tfrac{1}{2}ia;\tfrac{1}{2}\right)}
{1-2ia}=\frac{4\{X(a)+iY(a)\}}{\left|\Gamma\left(\frac{1}{4}+\frac{1}{2}ia\right)\right|^2}.
\end{multline}
Hence from (\ref{14}) we deduce that
\begin{equation} \label{14bb}
 \hat{N}_{W}'(a,0)
 =-\frac{4\sqrt{2}\pi e^{\pi a/2} X(a)}{\left|\Gamma\left(\frac{1}{4}
 +\frac{1}{2}ia\right)\right|^2},
\end{equation}
and consequently from (\ref{X1}) we arrive at (\ref{14b}), as asserted.
\end{proof}

Bringing everything together, we arrive at our desired numerically stable representations for the connection coefficients.
\begin{theorem} 
\label{thm2.7main}
Let $a \geq 0$. Then in (\ref{06}) we have
\begin{equation} \label{20}
c_{0}^{+}(a)=W(a,0)\left[
\sqrt {\pi} e^{\pi a/2}\left\{
1+\left(1+e^{-2\pi a}\right)^{-1/2}\right\}
-4Y(a)\left\{W'(a,0)\right\}^{2}
\right],
\end{equation}
\begin{multline} \label{21}
c_{0}^{-}(a)=-W(a,0) \bigg[
4Y(a)\left\{W'(a,0)\right\}^{2}   \\
+ \left. \sqrt {\pi} e^{-3\pi a/2}\left\{
1+e^{-2\pi a}+\left(1+e^{-2\pi a}\right)^{1/2}\right\}^{-1}
\right],
\end{multline}
and
\begin{equation} \label{20n}
c_{0}^{\pm}(-a)=W(a,0)\left[
4Y(a)\left\{W'(a,0)\right\}^{2}
\pm \sqrt {\pi} e^{-\pi a/2}\left\{
1 \pm \left(1+e^{2\pi a}\right)^{-1/2}\right\}
\right],
\end{equation}
where $W(a,0)$, $W'(a,0)$ and $Y(a)$ are given by (\ref{15}), (\ref{16}), and (\ref{Y}), respectively, and in the latter $F$ is the (unscaled) hypergeometric function defined by the rapidly converging series in (\ref{08}).
\end{theorem}

\begin{proof}
From (\ref{15}), (\ref{16}), (\ref{12cc}), and (\ref{X1})
\begin{equation} \label{18}
\frac{\hat{N}_{W}(a,0)W'(a,0)}{W(a,0)}
=4Y(a)\left\{W'(a,0)\right\}^{2}
-\frac{e^{\pi a}}{2\sqrt{\pi}}
\left|\Gamma\left(\tfrac{3}{4}+
\tfrac{1}{2}ia\right)\Gamma\left(\tfrac{1}{4}
+\tfrac{1}{2}ia\right)\right|.
\end{equation}
On using (\ref{reflect}) we have
\begin{equation} \label{19}
\frac{e^{\pi a}}{2\sqrt{\pi}}
\left|\Gamma\left(\tfrac{3}{4}+
\tfrac{1}{2}ia\right)\Gamma\left(\tfrac{1}{4}
+\tfrac{1}{2}ia\right)\right|
=e^{\pi a/2}\left(
\frac{\pi}{1+e^{-2\pi a}}\right)^{1/2},
\end{equation}
and thus from (\ref{17}), (\ref{14b}), (\ref{18}), and (\ref{19}) we get (\ref{20}) and (\ref{21}).

Next, for $a$ replaced by $-a$ we have that (\ref{15}) and (\ref{16}) are unchanged, and from (\ref{Y}) $Y(-a)=-Y(a)$. Hence from (\ref{20}) and (\ref{21}) we arrive at (\ref{20n}).
\end{proof}

The following shows for large $a$ that $c_{0}^{+}(a)$ is exponentially large, whereas $c_{0}^{-}(a)$ and $c_{0}^{\pm}(-a)$ are $\mathcal{O}(a^{-3/4})$.
\begin{corollary}
As $a \rightarrow \infty$
\begin{equation} \label{22c}
c_{0}^{+}(a)=\sqrt{2\pi} a^{-1/4}e^{\pi a/2}
\left\{1+\mathcal{O}(a^{-2})\right\},
\end{equation}
\begin{equation} \label{22d}
c_{0}^{-}(a)=-2^{-1/2}a^{-3/4}
\left\{1+\mathcal{O}(a^{-2})\right\},
\end{equation}
and
\begin{equation} \label{22dd}
c_{0}^{\pm}(-a)=2^{-1/2}a^{-3/4}
\left\{1+\mathcal{O}(a^{-2})\right\}.
\end{equation}
\end{corollary}
\begin{proof}
From (\ref{15}), (\ref{16}), and (\ref{stirling}) (or alternatively \cite[Eq. 5.11.12]{NIST:DLMF}) one can show that
\begin{equation} \label{22a}
W\left(a,0\right)=2^{-1/2}a^{-1/4}\left\{1
+\mathcal{O}(a^{-2})\right\},
\end{equation}
and
\begin{equation} \label{22b}
W'\left(a,0\right)=-2^{-1/2}a^{1/4}\left\{1
+\mathcal{O}(a^{-2})\right\},
\end{equation}
and hence using these along with (\ref{Y1}), (\ref{20}), (\ref{21}), and (\ref{20n}) we obtain the given approximations.
\end{proof}

\begin{remark}
From (\ref{05}) and (\ref{06}), and their differentiated forms, we obtain the interesting integrals
\begin{equation} \label{22eee}
\int_{0}^{\infty} W(a,\pm t)dt=\mp c_{0}^{\mp}(a),
\end{equation}
where $c_{0}^{\mp}(a)$ are given by \cref{thm2.7main}.

Also, from (\ref{15}), (\ref{17}), (\ref{14b}), and (\ref{22eee}), we note that
\begin{equation} \label{12b}
\int_{-\infty}^{\infty} W(a,t)dt
=-2\hat{N}_{W}'(a,0)W(a,0)
=2^{1/4}\sqrt {\pi} e^{\pi a/2}
\left|\frac{\Gamma\left(\tfrac{1}{4}+\tfrac{1}{2}ia\right)}
{\Gamma\left(\tfrac{3}{4}+
\tfrac{1}{2}ia\right)}\right|^{1/2}.
\end{equation}
\end{remark}

An interesting wrinkle in (\ref{06}) is that while it is numerically satisfactory for large $a$ and moderate to unbounded $x$, there is a large cancellation for large $a$ when $x$ is very small. This is on account of $N_{W}(a,x)$ and its first derivative vanishing at $x=0$, whereas the three terms on the RHS are exponentially large at $x=0$ for large $a$. To overcome this we wish to replace the three functions on the RHS of (\ref{06}) with alternative ones that mimic the behaviour $N_{W}(a,x)$ at $x=0$. To do so we note by setting $x=0$ in this equation and its derivative that
\begin{multline} \label{22u}
c_{0}^{+}(a)W(a,0)+c_{0}^{-}(a)W(a,0)-\hat{N}_{W}(a,0)
 \\
=c_{0}^{+}(a)W'(a,0)-c_{0}^{-}(a)W'(a,0)-\hat{N}_{W}'(a,0)=0.
\end{multline}
Thus our desired modification of (\ref{06}) is
\begin{multline} \label{22w}
N_{W}(a,x)=
c_{0}^{+}(a)\left\{W(a,x)-W(a,0)-W'(a,0) \,x\right\} \\
+c_{0}^{-}(a)\left\{W(a,-x)-W(a,0)+W'(a,0)\, x\right\} \\
-\left\{\hat{N}_{W}(a,x)-\hat{N}_{W}(a,0)-\hat{N}_{W}'(a,0)\, x\right\},
\end{multline}
which in conjunction with (\ref{15}), (\ref{16}), (\ref{12cc}), and (\ref{14b}) can be used if $a$ is large and $x$ is small. For all other values (\ref{06}) is stable.

We complete this section by obtaining a connection formula that can be used for complex argument $z$, and we only need to consider $\Re(z)\geq 0$ since $N_{W}(a,z)$ is even. Instead of $W(a,\pm x)$ we need the numerically satisfactory pair $W_{0}(a,z)$ and $W_{3}(a,z)$ as companions to $\hat{N}_{W}(a,z)$. Our result is as follows.

\begin{theorem} \label{thm2.9}
For $z \in \mathbb{C}$ let $W_{j}(a,z)$ ($j=0,3$) be defined by (\ref{02b}), and $\hat{N}_{W}(a,z)$ be given by (\ref{04}). For $x=z$ let $N_{W}(a,z)$ be given by (\ref{01}) and (\ref{22e}). Then these functions are related by 
\begin{equation} \label{22k}
N_{W}(a,z)=
c_{0}^{(0)}(a)W_{0}(a,z)+c_{0}^{(3)}(a)W_{3}(a,z)
-\hat{N}_{W}(a,z),
\end{equation}
where
\begin{equation} \label{22r}
c_{0}^{(0)}(a)=\frac {
\sqrt {\pi}2^{\frac{5}{4}+\frac{1}{2}ia}
e^{ \frac{1}{2}\pi a+\frac{3}{4}\pi i}Y(a)}
{\Gamma\left( \frac{1}{4}
-\frac{1}{2}ia \right) }
+2^{-\frac{5}{4}+\frac{1}{2}ia}
e^{ \frac{1}{2}\pi a+\frac{1}{4}\pi i}
\Gamma\left( \tfrac{1}{4}+\tfrac{1}{2}ia \right),
\end{equation}
$c_{0}^{(3)}(a)=\overline{c_{0}^{(0)}(a)}$, and $Y(a)$ is defined by (\ref{Y}).
\end{theorem}

\begin{proof}
Firstly from (\ref{22k}) we see that $c_{0}^{(3)}(a)=\overline{c_{0}^{(0)}(a)}$, since when $z=x \in \mathbb{R}$ both $N_{W}(a,x)$ and $\hat{N}_{W}(a,x)$ are real, and $W_{0}(a,x)=\overline{W_{3}(a,x)}$. Next, on recalling that $N_{W}(a,0)=N_{W}'(a,0)=0$, we have from (\ref{22k}) and its derivative
\begin{equation} \label{22l}
c_{0}^{(0)}(a)=
\frac{\hat{N}_{W}(a,0)W_{3}'(a,0)
-\hat{N}_{W}'(a,0)W_{3}(a,0)}
{\mathscr{W}\left\{W_{0},W_{3}\right\}(a,0)}.
\end{equation}

Now from (\ref{02b}) and \cite[Eq. 12.2.12]{NIST:DLMF} the required Wronskian is given by
\begin{equation} \label{22m}
\mathscr{W}\left\{W_{0}(a,z),W_{3}(a,z)\right\}
=-ie^{-\pi a/2},
\end{equation}
and from (\ref{02b}) and \cite[Eqs. 12.2.6, 12.2.7]{NIST:DLMF} we have
\begin{equation} \label{22n}
W_{3}(a,0)=\overline{W_{0}(a,0)}=
\frac{\sqrt{\pi}}{2^{\frac{1}{4}-\frac{1}{2}ia} \Gamma\left( \frac{3}{4}-\frac{1}{2}ia \right)},
\end{equation}
and
\begin{equation} \label{22p}
W_{3}'(a,0)=\overline{W_{0}'(a,0)}=
\frac{\sqrt{\pi}e^{-3\pi i/4}2^{\frac{1}{4}+\frac{1}{2}ia}}{ \Gamma\left( \frac{1}{4}-\frac{1}{2}ia \right)}.
\end{equation}
Therefore using \cref{lem2.4,lem2.6}, (\ref{22l}) - (\ref{22p}), along with the gamma function reflection formula (\ref{reflect}), we arrive at (\ref{22r}). 
\end{proof}

In (\ref{22r}) the first term dominates when $a$ is large since it is exponentially large in comparison to the second term, and from (\ref{stirling}) and (\ref{Y1}) we then get
\begin{equation} \label{22s}
c_{0}^{(0)}(a)=
2^{-1/2} a^{-\frac{3}{4}+\frac{1}{2}ia}
e^{\frac{3}{4}\pi a + \frac{5}{8}\pi i - \frac{1}{2}ia}
\left\{1+\mathcal{O}(a^{-1})\right\},
\end{equation}
with the corresponding approximation for $c_{0}^{(3)}(a)$ being the complex conjugate of this.

If $a$ is large and $|z|$ is small then similarly to (\ref{22w}) one should use in place of (\ref{22k}) the alternative form
\begin{multline} \label{22v}
N_{W}(a,z)=
c_{0}^{(0)}(a)\left\{W_{0}(a,z)-W_{0}(a,0)-W_{0}'(a,0)z\right\} \\
+c_{0}^{(3)}(a)\left\{W_{3}(a,z)-W_{3}(a,0)-W_{3}'(a,0)z\right\} \\
-\left\{\hat{N}_{W}(a,z)-\hat{N}_{W}(a,0)-\hat{N}_{W}'(a,0)z\right\}.
\end{multline}

\section{Uniform asymptotic expansions} 
\label{sec3}
Here we record asymptotic expansions for the Nield-Kuznetsov functions for large $a$ that are uniformly valid for all real and complex values of the argument $z$. We extract the required results from \cite[Sect. 4.1]{Dunster:2021:UAP}, and refer the reader there for proofs of the foregoing results. We first must define terms that are used, and the notation here differs slightly from the above reference.

It is convenient to write $a=\tfrac{1}{2} u$ and rescale the argument from $z$ to $\sqrt{2u} \, z$. Thus we shall focus primarily on the function $W_{R}^{(0,3)}\left(\tfrac{1}{2}u,\sqrt{2u} \, z\right)$. Although we have so far considered $R=0$ (which is the complementary Nield-Kuznetsov function) we also shall include $R=1$, since both of these will be used in the proceeding section on Laplace transforms.

Firstly, with $z$ rescaled as above, let
\begin{equation} \label{23}
\beta=\frac{z}{\sqrt{1-z^{2}}},
\end{equation}
where the branch of the square root is positive for $-1<z<1$ and is continuous in the plane having a cuts along $(-\infty,-1]$ and $[1,\infty)$. Note $\beta \rightarrow \pm i$ as $z \rightarrow \infty$ in the upper and lower half-planes, respectively.

Next define
\begin{equation} \label{zeta1}
\xi=\tfrac{2}{3}\zeta^{3/2}
=\tfrac{1}{2}\arccos(z)-\tfrac{1}{2}z\sqrt{1-z^{2}}.
\end{equation}
The branch in (\ref{zeta1}) is chosen so that $\xi \geq 0$ and $\zeta \geq 0$ for $-1\leq z \leq 1$, and by continuity elsewhere in the $z$ plane having a cut along $(-\infty, -1]$ (as well as $[1,\infty)$ for $\xi$). Hence $\zeta \leq 0$ for $1\leq z < \infty$, and is given by
\begin{equation} \label{zeta2}
\tfrac{2}{3}(-\zeta)^{3/2}=\tfrac{1}{2}z\sqrt {z^{2}-1}-\tfrac{1}{2}\ln\left( z+\sqrt {z^{2}-1} \right).
\end{equation}
We remark that $\zeta$ (unlike $\xi$) is an analytic function of $z$ for $\Re(z) \geq 0$, in particular at $z=1$.

We now define an easily computed sequence of polynomials via
\begin{equation} \label{24}
\mathrm{E}_{1}(\beta)=\tfrac{1}{24}\beta
\left(5\beta^{2}+6\right),
\end{equation}
\begin{equation} \label{25}
\mathrm{E}_{2}(\beta)=
\tfrac{1}{16}\left(\beta^{2}+1\right)^{2} 
\left(5\beta^{2}+2\right),
\end{equation}
and for $s=2,3,4\ldots$
\begin{equation} \label{26}
\mathrm{E}_{s+1}(\beta) =
\frac{1}{2} \left(\beta^{2}+1 \right)^{2}\mathrm{E}_{s}^{\prime}(\beta)
+\frac{1}{2}\int_{i\sigma(s)}^{\beta}
\left(p^{2}+1 \right)^{2}
\sum\limits_{j=1}^{s-1}
\mathrm{E}_{j}^{\prime}(p)
\mathrm{E}_{s-j}^{\prime}(p) dp.
\end{equation}
Here $\sigma(s)=1$ for $s$ odd and $\sigma(s)=0$ for $s$ even, so that the even and odd coefficients are even and odd functions of $\beta$, respectively, with $\mathrm{E}_{2s}(\pm i)=0$ ($s=1,2,3,\ldots$). 

We further define two sequences $\left\{a_{s}\right\}_{s=1}^{\infty}$ and $\left\{\tilde{a}_{s}\right\}_{s=1}^{\infty}$ by $a_{1}=a_{2}=\frac{5}{72}$, $\tilde{a}_{1}=\tilde{a}_{2}=-\frac{7}{72}$, with subsequent terms $a_{s}$ and $\tilde{{a}}_{s}$ ($s=2,3,\ldots $) satisfying the same recursion formula
\begin{equation} \label{27}
b_{s+1}=\frac{1}{2}\left(s+1\right) b_{s}+\frac{1}{2}
\sum\limits_{j=1}^{s-1}{b_{j}b_{s-j}}.
\end{equation}

With the above definitions we then introduce a sequence of coefficient functions
\begin{equation} \label{28}
\mathcal{E}_{s}(z) =\mathrm{E}_{s}(\beta) +
(-1)^{s}a_{s}s^{-1}\xi^{-s},
\end{equation}
and
\begin{equation} \label{29}
\tilde{\mathcal{E}}_{s}(z) =\mathrm{E}_{s}(\beta)
+(-1)^{s}\tilde{a}_{s}s^{-1}\xi^{-s}.
\end{equation}

Next $\mathcal{A}(u,z)$ and $\mathcal{B}(u,z)$ are two coefficient functions which are analytic in the cut plane $\mathbb{C} \setminus (-\infty,-1]$, and as $u \rightarrow \infty$, uniformly for $z$ bounded away by a distance $\delta$ from this cut, they possess the asymptotic expansions
\begin{equation} \label{31}
\mathcal{A}(u,z) \sim \phi(z)\exp \left\{ \sum\limits_{s=1}^{\infty}\frac{
\tilde{\mathcal{E}}_{2s}(z) }{u^{2s}}\right\} \cosh \left\{ \sum\limits_{s=0}^{\infty}\frac{\tilde{\mathcal{E}}_{2s+1}(z) }{u^{2s+1}}\right\},
\end{equation}
and
\begin{equation} \label{32}
\mathcal{B}(u,z) \sim \frac{\phi(z)}{u^{1/3}\sqrt{\zeta}}\exp \left\{ \sum\limits_{s=1}^{\infty}\frac{
\mathcal{E}_{2s}(z) }{u^{2s}}\right\} \sinh \left\{ \sum\limits_{s=0}^{\infty}\frac{\mathcal{E}_{2s+1}(z) }{u^{2s+1}}\right\},
\end{equation}
where
\begin{equation} \label{phi}
\phi(z)=\left(\frac{\zeta}{1-z^2}\right)^{1/4}.
\end{equation}
Note that $\phi(z)$ has a removable singularity at $z=1$ since $\zeta$, as defined by (\ref{zeta1}), has a simple zero at this point. We also remark that $z=1$ is a turning point from the differential equation from which these expansions were derived.

The coefficients $\mathcal{E}_{2s}(z)$ and $\tilde{\mathcal{E}}_{2s}(z)$ are not analytic at  $z=1$. Thus these expansions (or rather a truncated series in (\ref{31}) and (\ref{32})) cannot be used directly near this singularity. However both $\mathcal{A}(u,z)$ and $\mathcal{B}(u,z)$ are themselves analytic at this point, and there are two ways to compute them near and at this point, described next.

If many terms are required for high accuracy, Cauchy's integral formula can be used, as given in \cite[Eqs. (19) and (20)]{Dunster:2020:ASI} and \cite[Thm. 4.2]{Dunster:2021:SEB}. Basically the expansions for $\mathcal{A}(u,z)$ and $\mathcal{B}(u,z)$ are inserting in the Cauchy integral representation of these functions, and being valid on the contour of integration that encloses $z=1$ this allows accurate and stable computation of each function.

If only a few terms are required we can expand these functions into a traditional asymptotic expansion involving inverse powers of $u^2$. Specifically, for $\mathcal{A}(u,z)$ we find from (\ref{31}) as $u \rightarrow \infty$
\begin{equation} \label{A1}
\mathcal{A}(u,z) \sim
\phi(z) \sum\limits_{s=0}^{\infty}\frac{\mathrm{A}_{s}(z) }{u^{2s}},
\end{equation}
where $\mathrm{A}_{0}(z)=1$,
\begin{equation} \label{A3}
\mathrm{A}_{1}(z)
=\tfrac{1}{2}\left\{\tilde{\mathcal{E}}_{1}^{2}(z)
+2\tilde{\mathcal{E}}_{2}(z)\right\},
\end{equation}
\begin{equation} \label{A4}
\mathrm{A}_{2}(z)
=\tfrac{1}{24}\left\{\tilde{\mathcal{E}}_{1}^{4}(z)
+12\tilde{\mathcal{E}}_{1}^{2}(z)\tilde{\mathcal{E}}_{2}(z)
+24\tilde{\mathcal{E}}_{1}(z)\tilde{\mathcal{E}}_{3}(z)
+12\tilde{\mathcal{E}}_{2}^{2}(z)
+24\tilde{\mathcal{E}}_{4}(z)\right\},
\end{equation}
and so on. 

Similarly from (\ref{32}) we have
\begin{equation} \label{B1}
\mathcal{B}(u,z) \sim
\frac{\phi(z)}{u^{4/3}}
\sum\limits_{s=0}^{\infty}\frac{\mathrm{B}_{s}(z) }{u^{2s}},
\end{equation}
with the first two terms being
\begin{equation} \label{B2}
\mathrm{B}_{0}(z)=
\zeta^{-1/2}\mathcal{E}_{1}(z),
\end{equation}
and
\begin{equation} \label{B3}
\mathrm{B}_{1}(z)=
\zeta^{-1/2}
\left\{\tfrac{1}{6}\mathcal{E}_{1}^{3}(z)
+\mathcal{E}_{1}(z)\mathcal{E}_{2}(z)
+\mathcal{E}_{3}(z)
\right\}.
\end{equation}

From the general turning point theory of differential equations \cite[Chap. 11]{Olver:1997:ASF} we know that these coefficients must be defined and analytic at $z=1$, even though $\tilde{\mathcal{E}}_{s}(z)$ are not. The singularities at $z=1$ cancel out in each of $\mathrm{A}_{s}(z)$ and $\mathrm{B}_{s}(z)$, rendering them analytic at the turning point. Thus a Taylor series can be employed very close to the removable singularity. For example, as $z \rightarrow 1$ we find with the aid of Maple\footnote{Maple 2020. Maplesoft, a division of Waterloo Maple Inc., Waterloo, Ontario} from (\ref{23}) - (\ref{29}), (\ref{A3}), (\ref{A4}), (\ref{B2}), and (\ref{B3})
\begin{equation} \label{A5}
\mathrm{A}_{1}(z)=
\frac{7}{900}-\frac {6849}{616000}(z-1)
+\mathcal{O}\left\{(z-1)^2\right\},
\end{equation}
\begin{equation} \label{A6}
\mathrm{A}_{2}(z)=
\frac{72846269}{13970880000}-\frac {1740991}{149600000}(z-1)
+\mathcal{O}\left\{(z-1)^2\right\},
\end{equation}
\begin{equation} \label{B4}
\mathrm{B}_{0}(z)=
-\frac{9}{280}2^{1/3}+\frac {7}{450}2^{1/3}(z-1)
+\mathcal{O}\left\{(z-1)^2\right\},
\end{equation}
and
\begin{equation} \label{B5}
\mathrm{B}_{1}(z)=
-\frac{3013}{260000}2^{1/3}+\frac {5549641}{517440000}2^{1/3}(z-1)
+\mathcal{O}\left\{(z-1)^2\right\}.
\end{equation}

In a more general setting, analyticity of asymptotic coefficients in this and similar situations follows from the following result, and we will use this later.

\begin{theorem} \label{thm:nopoles}
Let $0<\rho_{1}<\rho_{2}$, $u>0$, $z_{0}\in \mathbb{C}$, $H(u,z)$ be an analytic function of $z$ in the open disk $D=\{z:\, |z-z_{0}|<\rho_{2}\}$, and $h_{s}(z)$ ($s=0,1,2,\ldots$) be a sequence of functions that are analytic in $D$ except possibly for an isolated singularity at $z=z_{0}$. If $H(u,z)$ is known to possess the asymptotic expansion
\begin{equation} \label{301}
H(u,z) \sim \sum\limits_{s=0}^{\infty}\frac{h_{s}(z)}{u^{s}}
\quad (u \rightarrow \infty),
\end{equation}
in the annulus $\rho_{1}<|z-z_{0}|<\rho_{2}$, then $z_{0}$ is an ordinary point or a removable singularity for each $h_{s}(z)$, and the expansion (\ref{301}) actually holds for all $z \in D$ (with $h_{s}(z_{0})$ defined by the limit of this function at $z_{0}$ if it is a removable singularity).
\end{theorem}

\begin{proof}
To establish both assertions we shall use certain Cauchy integrals. Firstly, since $z_{0}$ is an isolated singularity of $h_{s}(z)$ for each $s$ it is expressible as a Laurent series at this point, of the form
\begin{equation} \label{304}
h_{s}(z) = \sum\limits_{j=-\infty}^{\infty}
a_{s,j}\left(z-z_{0}\right)^{j} \quad (0<|z-z_{0}|<\rho_{2}),
\end{equation}
for some coefficients $a_{s,j}$.

Now choose any $\rho'$ such that $\rho_{1}<\rho'<\rho_{2}$, and let $\Gamma$ be a positively orientated circle centered at $z=z_{0}$ of radius $\rho'$. By hypothesis (\ref{301}) certainly holds for all points on this contour, and since $\Gamma\in D$ we have for arbitrary positive integer $n$ and nonnegative integer $r$, from the Cauchy-Goursat theorem and analyticity of $H(u,z)$
\begin{equation} \label{302}
0= \oint_{\Gamma} \left(z-z_{0}\right)^{r} H(u,z) dz
= \sum\limits_{s=0}^{n-1} 
\frac{h_{r,s}}{u^{s}}
+\mathcal{O}\left(\frac{1}{u^{n}}\right),
\end{equation}
where
\begin{equation} \label{303}
h_{r,s} = \oint_{\Gamma} \left(z-z_{0}\right)^{r} h_{s}(z) dz.
\end{equation}
Now by uniqueness of asymptotic series it follows from (\ref{302}) for each $s$ that $h_{r,s} =0$ for all nonnegative $r$. But from (\ref{304}), (\ref{303}) and the residue theorem we have $h_{r,s} = 2 \pi i a_{s,-r-1}$ ($r=0,1,2,\ldots$), and hence all negative powers of $(z-z_{0})$ in the Laurent expansion (\ref{304}) vanish. Thus each function $h_{s}(z)$ is either analytic at $z=z_{0}$, or has a removable singularity there, as asserted; if the latter, define $h_{s}(z_{0})=a_{s,0}$. 

Having established each $h_{s}(z)$ is analytic in $D$ we can now employ Cauchy's integral formula, and for any point $z$ interior to $\Gamma$, again using (\ref{301}) for points on the contour, we have
\begin{multline} \label{305}
H(u,z)=\frac{1}{2\pi i}
\oint_{\Gamma}
\frac{H(u,t) }{t-z} dt 
= \sum\limits_{s=0}^{n-1}\frac{1}{2 \pi i u^{s}} 
\oint_{\Gamma}\frac{h_{s}(t)}{t-z} dt
+\mathcal{O}\left(\frac{1}{u^{n}}\right) \\
=\sum\limits_{s=0}^{n-1} 
\frac{h_{s}(z)}{u^{s}}
+\mathcal{O}\left(\frac{1}{u^{n}}\right).
\end{multline}
This verifies that (\ref{301}) holds for all $z \in D$. 
\end{proof}

Next, there are some constants and other coefficients we require. Firstly, we have
\begin{equation} \label{34}
\gamma_{0}(u)=
2^{-\frac{3}{2}+\frac{1}{4}iu}\pi^{-1/2}
u^{-13/12}e^{\frac{1}{8}\pi u+i(\chi(u)+\frac{1}{8}\pi )}
\Gamma\left(\tfrac{1}{4}+\tfrac{1}{4}iu \right),
\end{equation}
and
\begin{equation} \label{34.1}
\gamma_{1}(u)=
2^{-\frac{1}{2}+\frac{1}{4}iu}\pi^{-1/2}
u^{-19/12}e^{\frac{1}{8}\pi u+i(\chi(u)-\frac{1}{8}\pi )}
\Gamma\left(\tfrac{3}{4}+\tfrac{1}{4}iu \right).
\end{equation}
In these $\chi(u)$ is another constant that is defined by a certain asymptotic series and error term (which has an explicit bound). For brevity we omit details of the error bound here, and instead note that this constant possesses the following asymptotic expansion as $u \rightarrow \infty$

\begin{equation} \label{41}
\chi(u) \sim
\frac{u}{4}\ln\left(\frac{2e}{u}\right)
+ \frac{1}{u} \sum\limits_{s=0}^{\infty}
\frac{\mathrm{e}_{s}}{u^{2s}},
\end{equation}
where $\mathrm{e}_{s}$ are the coefficients in the following relation (which can be derived using (\ref{stirling}))
\begin{equation} \label{42}
\frac{1}{2}
\Im\left\{ \ln\left(\Gamma\left(\tfrac{1}{2}-\tfrac{1}{2}iu \right) \right)\right\} \sim
\frac{u}{4}\ln\left(\frac{2e}{u}\right)
+ \frac{1}{u} \sum\limits_{s=0}^{\infty}
\frac{\mathrm{e}_{s}}{u^{2s}}.
\end{equation}
The first two coefficients are $\mathrm{e}_{0}=-\frac{1}{24}$ and $\mathrm{e}_{1}=-\frac{7}{720}$.

Incidentally, one can show from (\ref{41}) and (\ref{42}) that for large $u$ and arbitrary positive integer $n$
\begin{equation} \label{92}
e^{i\chi(u)}
=\left\{\frac{\Gamma\left(\tfrac{1}{2}-\tfrac{1}{2}iu \right)}{\Gamma\left(\tfrac{1}{2}+\tfrac{1}{2}iu \right)}\right\}^{1/4}\left\{1+\mathcal{O}\left(\frac{1}{u^{n}}\right)\right\},
\end{equation}
and we shall use this later.

We note in passing that from (\ref{stirling}), (\ref{34}) - (\ref{42}) one finds
\begin{equation}
\label{eq34a}
\gamma_{0}(u)=
\frac {1}{\sqrt{2}u^{4/3}}
\left\{1+\frac {1}{8{u}^{2}}+\frac {41}{128{u}^{4}}
+\mathcal{O}\left(\frac {1}{{u}^6}\right)\right\},
\end{equation}
and
\begin{equation}
\label{eq34d}
\gamma_{1}(u)=
\frac {1}{\sqrt{2}u^{4/3}}
\left\{1-\frac {1}{8{u}^{2}}-\frac {39}{128{u}^{4}}
+\mathcal{O}\left(\frac {1}{{u}^6}\right)\right\},
\end{equation}
as $u \rightarrow \infty$

Next for $R=0,1$ let
\begin{equation} \label{39}
G_{0,R}(z)=z^{R}/\left(z^{2}-1\right),
\end{equation}
and
\begin{equation} \label{40}
G_{s+1,R}(z)=G''_{s,R}(z)/\left(1-z^{2}\right)
\quad (s=0,1,2,\ldots ).
\end{equation}
For $R=0$ and $R=1$ a third slowly-varying coefficient function $\mathcal{G}_{R}(u,z)$, companion to (\ref{31}) and (\ref{32}) and valid for the same values of $z$, is then defined such that for $u \rightarrow \infty$
\begin{equation}
\label{eq35a}
\mathcal{G}_{R}(u,z) \sim \frac{1}{u^{2}}\sum\limits_{s=0}^{\infty} \frac{G_{s,R}(z)}{u^{2s}} -\frac{\gamma_{R}(u)J(u,z)}{u^{2/3}\zeta }\left(\frac{\zeta}{1-z^2}\right)^{1/4},
\end{equation}
where
\begin{multline} \label{36}
 J(u,z) \sim -\exp \left\{\sum\limits_{s=1}^{\infty} \frac{\tilde{\mathcal{E}}_{2s}(z)}{u^{2s}} \right\}\cosh \left\{\sum\limits_{s=0}^{\infty} 
\frac{\tilde{\mathcal{E}}_{2s+1}(z)}{u^{2s+1}}  
\right\}\sum\limits_{k=0}^{\infty} \frac{(3k)!}{k!\left( 3u^{2}\zeta^{3} \right)^{k}} \\ 
 +\frac{1}{u\zeta^{3/2}}\exp \left\{\sum\limits_{s=1}^{\infty} \frac{{\mathcal{E}}_{2s}(z)}{u^{2s}}  \right\}\sinh \left\{\sum\limits_{s=0}^{\infty} 
\frac{{\mathcal{E}}_{2s+1}(z)}{u^{2s+1}}  \right\}\sum\limits_{k=0}^{\infty}
\frac{(3k+1)!}{k!\left(3u^{2}\zeta^{3} \right)^{k}}.
\end{multline}

The function $\mathcal{G}_{R}(u,z)$ is analytic at $z=1$, but the truncated series (\ref{eq35a}) and (\ref{36}) are not and so break down near this point. An asymptotic expansion for $\mathcal{G}_{R}(u,z)$ can be computed near this point in similar way to $\mathcal{A}(u,z)$ and $\mathcal{B}(u,z)$, as described above. That is, either by Cauchy's integral formula, or by re-expanding the series in inverse powers of $u$ for a few terms, similarly to (\ref{A1}) and (\ref{B1}). In the latter case the coefficients are analytic at $z=1$ from \cref{thm:nopoles}.

Having defined all the terms required, from \cite[Sect. 4.1]{Dunster:2021:UAP} we then have the following main result, which in conjunction with (\ref{04}) provides a uniform asymptotic expansion for $\hat{N}_W\left(\tfrac{1}{2}u,\sqrt{2u} \, z\right)$ for the case $R=0$ (and likewise (\ref{38}) below).
\begin{theorem} \label{thm3}
For $R=0,1$
\begin{multline} \label{34b}
W_{R}^{(0,3)}\left(\tfrac{1}{2}u,\sqrt{2u} \, z\right)
=(2u)^{\frac{1}{2}R+1} \Big[
\mathcal{G}_{R}(u,z) \\
\left. + \pi \gamma_{R}(u)
\left\{\mathrm{Hi}\left(u^{2/3}\zeta  \right)\mathcal{A}(u,z) +  \mathrm{Hi}'\left(u^{2/3}\zeta \right)\mathcal{B}(u,z) \right\} 
\right],
\end{multline}
where $\mathrm{Hi}(z)$ is the Scorer function defined by
\begin{equation}
\label{eqHi}
\mathrm{Hi}(z)=\frac{1}{\pi }\int_{0}^{\infty} \exp  \left(-\tfrac{1}{3} t^{3}+z t \right)dt,
\end{equation}
$\zeta$ is defined by (\ref{zeta1}), and $\mathcal{A}(u,z)$, $\mathcal{B}(u,z)$, and $\mathcal{G}_{R}(u,z)$ possess the asymptotic expansions (\ref{31}), (\ref{32}), and (\ref{eq35a}), respectively, as $u \rightarrow \infty$ uniformly for $z$ bounded away by a distance $\delta$ from the cut $(-\infty,-1]$ (suitably modified as described after (\ref{phi}) if $z$ is close to or equal to 1).
\end{theorem}

\begin{remark}
$\mathrm{Hi}(z)$ is the uniquely defined particular solution of the inhomogeneous Airy equation
\begin{equation}
\label{eqHiEq}
\frac{d^{2}w}{dz^{2}}-zw=\frac{1}{\pi},
\end{equation}
having the behavior
\begin{equation}
\label{HiZ}
\mathrm{Hi}(z)\sim -\frac{1}{\pi z} \quad \left(z\to \infty,\; \left| \arg(-z) \right|\le \tfrac{2}{3}\pi -\delta \right).
\end{equation}
This function grows exponentially as $z \rightarrow \infty$ in $|\arg(z)|\le \tfrac{1}{3}\pi -\delta$ (see \cite[Eq. 9.12.29]{NIST:DLMF}, and also \cite[Sect. 9.12]{NIST:DLMF} for further properties.)
\end{remark}

As $z \rightarrow \infty$ it follows from \cref{thm3} that $W_{R}^{(0,3)}(a,z)=\mathcal{O}(z^{-2+R})$ for $|\arg(z)| \\   \leq \tfrac{1}{2}\pi - \delta$, and is exponentially large as $z \rightarrow \infty$ for $\delta < |\arg(-z)| \leq \tfrac{1}{2}\pi - \delta$. No other solution of (\ref{02d}) shares this recessive property in parts of \emph{both} the first and fourth quadrants.

\begin{figure}[htbp]
  \centering
  \includegraphics[trim={0 100 0 0},width=0.8\textwidth,keepaspectratio]{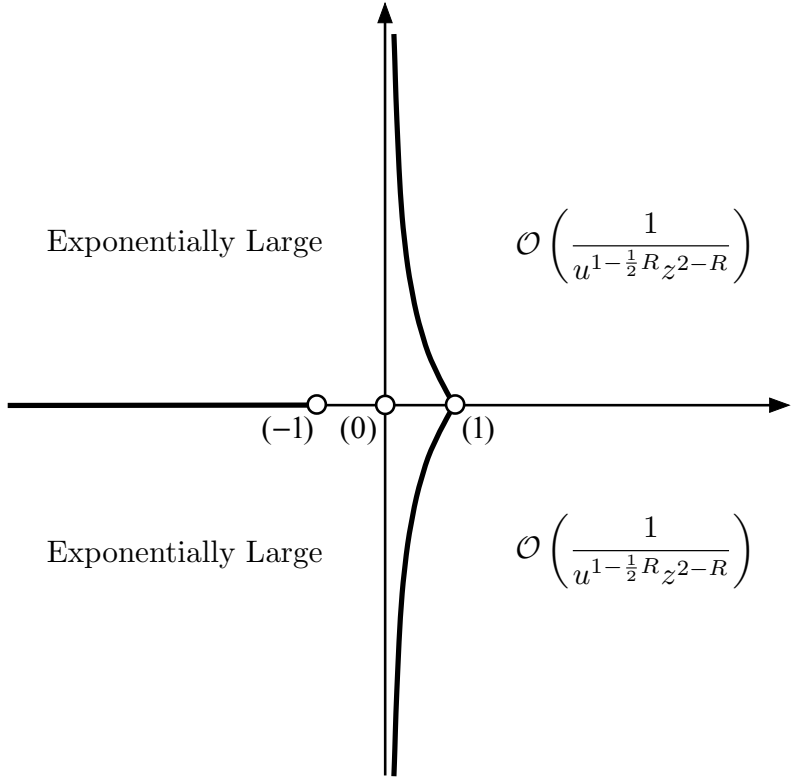}
  \caption{Regions of subdominance and dominance for $W_{R}^{(0,3)}(\tfrac{1}{2}u,\sqrt{2u} \, z)$}
  \label{fig:fig1}
\end{figure}

Moreover we have that for large $u$
\begin{equation} \label{38}
W_{R}^{(0,3)}\left(\tfrac{1}{2}u,\sqrt{2u} \, z\right)
\sim \frac{(2u)^{\frac{1}{2}R+1}}{u^{2}}\sum\limits_{s=0}^{\infty} \frac{G_{s,R}(z)}{u^{2s}}
=\mathcal{O}\left(\frac{1}{u^{1-\frac{1}{2}R} z^{2-R}}\right),
\end{equation}
uniformly for $z$ lying in "most" of the right half-plane. On the other hand this function is exponentially large in $u$ in a part of the right half-plane where $\Re(z)<1$, as well as in the whole of left half-plane. More precisely, the boundary that separates these two regions is depicted in \cref{fig:fig1}, and consists of two level curves $\Re(\xi)=0$ in the first and fourth quadrants, emanating from $z=1$, where $\xi$ is given by (\ref{zeta1}). The simple expansion (\ref{38}) can be considered as a special case of (\ref{34b}) that is uniformly valid in a more restricted region, namely for all unbounded $z$ to the right of the curves $\Re(\xi)=0$ and that lie at least a distance $\delta$ from them.

We finish with expansions for homogeneous solutions. We use the standard notation for Airy functions of complex argument $\mathrm{Ai}_{l}(z)=\mathrm{Ai}(z e^{-2\pi i l/3})$ ($l=0,\pm 1$), and with these we define the three functions 
\begin{equation} \label{30}
w_{l}(u,z) =\mathrm{Ai}_{l}\left(u^{2/3}\zeta \right) 
\mathcal{A}(u,z) +\mathrm{Ai}_{l}^{\prime }\left(u^{2/3}\zeta\right) \mathcal{B}(u,z).
\end{equation}
Then we have for unbounded $z$ as described in \cref{thm3}
\begin{equation} \label{41a}
W_{j}\left(\tfrac{1}{2}u,\sqrt {2u} z\right)
=2^{3/4}\sqrt{\pi}u^{-1/12}
e^{-\pi u/8} e^{\pm i (\chi(u)
+\frac{1}{24}\pi)} w_{\mp 1}(u,z),
\end{equation}
with upper signs for $j=0$ and lower signs for $j=3$.

Recall for large $z$ that $W_{0}(\frac{1}{2}u,\sqrt {2u} z)$, $W_{3}(\frac{1}{2}u,\sqrt {2u} z)$ and $W_{R}^{(0,3)}(\frac{1}{2}u,\sqrt {2u} z)$ form a numerically satisfactory set of functions in the subdominant region depicted in \cref{fig:fig1}, since $W_{0}(\frac{1}{2}u,\sqrt {2u} z)$ is recessive (exponentially small at infinity) in the first quadrant, $W_{3}(\frac{1}{2}u,\sqrt {2u} z)$ is recessive (exponentially small at infinity) in the fourth quadrant, and $W_{R}^{(0,3)}(\frac{1}{2}u,\sqrt {2u} z)$ is the unique particular solution that is subdominant in the above described region.

With these three fundamental solutions, an asymptotic expansion for $N_{W}(a,z)$ follows from  \cref{thm2.9,thm3}, (\ref{30}) and (\ref{41a}). This is certainly uniformly valid for $\Re(z) \geq 0$, with of course the left half-plane immediately being covered by the evenness of $N_{W}(a,z)$.

We also remark that an asymptotic expansion for $\hat{N}_{W}(\tfrac{1}{2}u,\sqrt{2u} \, z)$ valid in the whole left half-plane, and in particular including the cut $-\infty < z \leq -1$ excluded from (\ref{34b}), can be obtained from \cref{thm3} along with (\ref{22t}), (\ref{30}) and (\ref{41a}).

\section{Laplace Transforms of \texorpdfstring{$W(a,\pm t)$}{Lg}}
\label{sec4}
We consider the Laplace transform of the Weber functions defined by
\begin{equation} \label{60}
L_{W}^{\pm}(a,\lambda)  =\int_{0}^{\infty} e^{-\lambda t} W(a,\pm t)dt \quad \left(|\arg(\lambda)| \leq \tfrac{1}{2}\pi \right),
\end{equation}
and by analytic continuation for other values of $\arg(\lambda)$. As we shall show, both are entire functions of $\lambda$.

For bounded $a$ and large $\lambda$ the asymptotics are simple to obtain, and read as follows.
\begin{theorem} \label{thm4.1}
For fixed $a$ and $\lambda \rightarrow \infty$ with $|\arg(\lambda)| \leq \frac{1}{2}\pi - \delta$
\begin{equation} \label{60d}
L_{W}^{\pm}(a,\lambda) \sim
\frac{W(a,0)}{\lambda}
\sum_{ s=0}^{\infty} \frac{\alpha_{s}(a)}{\lambda^{2s}}
\pm \frac{W'(a,0)}{\lambda^{2}}
\sum_{ s=0}^{\infty} \frac{\beta_{s}(a)}{\lambda^{2s}},
\end{equation}
where $\alpha_{0}(a)=\beta_{0}(a)=1$, $\alpha_{1}(a)=\beta_{1}(a)=a$, and for $ s=0,1,2,\ldots$ 
\begin{equation}\label{60e}
\alpha_{s+2}=a\alpha_{s+1}-\tfrac{1}{2}(s+1)(2s+1)\alpha_{s},
\end{equation}
\begin{equation} \label{60f}
\beta_{s+2}=a\beta_{s+1}-\tfrac{1}{2}(s+1)(2s+3)\beta_{s},
\end{equation}
where $W(a,0)$ and $W'(a,0)$ are given by (\ref{15}) and (\ref{16}).
\end{theorem}

\begin{proof}
We have by Watson's lemma \cite[Chap. 4, Thm. 3.2]{Olver:1997:ASF} applied to (\ref{60})
\begin{equation} \label{60b}
L_{W}^{\pm}(a,\lambda) \sim
\sum_{s=0}^{\infty} (\pm 1)^{s} \frac{a_{s}s!}{\lambda^{s+1}},
\end{equation}
for $|\arg(\lambda)| \leq \frac{1}{2}\pi - \delta$, where $a_{s}$ are the coefficients in the asymptotic expansion
\begin{equation} \label{60c}
W(a,t) \sim
\sum_{s=0}^{\infty}a_{s}t^{s} \quad (t \rightarrow 0).
\end{equation}
These coefficients are given by \cite[Eqs. 12.14.8 - 12.14.12]{NIST:DLMF} (actually (\ref{60c}) is a convergent series), and inserting them into (\ref{60b}) yields (\ref{60d}) - (\ref{60f}).
\end{proof}

We now consider the more difficult problem of obtaining asymptotic expansions for large $a$ that are uniformly valid for all real and complex $\lambda$, including the analytic continuation of (\ref{60}) to the left half $\lambda$ plane where the integral diverges. To this end we begin with the following.
\begin{lemma} \label{lem4.2}
Each $L_{W}^{\pm}(a,\lambda)$ is an entire function of $\lambda$ and satisfies the differential equation
\begin{equation} \label{60a}
\frac{\partial^2  L_{W}^{\pm}(a,\lambda) }{\partial \lambda^2}
 =4\left(a-\lambda^2\right)L_{W}^{\pm}(a,\lambda)  \pm 4W'(a,0)+4\lambda W(a,0).
\end{equation}
\end{lemma}

\begin{proof}
Assume temporarily that $|\arg(\lambda)| \leq \frac{1}{2}\pi - \delta$. Differentiation of (\ref{60}) twice with respect to $\lambda$ gives
\begin{equation} \label{61}
\frac{\partial^2  L_{W}^{\pm}(a,\lambda) }{\partial \lambda^2}  =\int_{0}^{\infty} t^2 e^{-\lambda t} W(a,\pm t)dt.
\end{equation}
Then from the Weber differential equation (\ref{02a}) we find that
\begin{equation} \label{62}
\frac{\partial^2  L_{W}^{\pm}(a,\lambda) }{\partial \lambda^2}
 =4a\int_{0}^{\infty} e^{-\lambda t} W(a,\pm t)dt
 -4\int_{0}^{\infty} e^{-\lambda t} W''(a,\pm t)dt.
\end{equation}
Next from integration by parts twice of the second integral on the RHS of (\ref{62}), and referring to (\ref{60}), we then get (\ref{60a}). Finally, by analytic continuation the restriction on $\arg(\lambda)$ can be relaxed and $L_{W}^{\pm}(a,\lambda)$ are seen to be entire functions of $\lambda$ since the differential equations (\ref{60a}) they satisfy have no finite singularities.
\end{proof}

As a consequence the following explicit representation can be inferred.
\begin{theorem}
The Laplace transforms (\ref{60}) are expressible in terms of the general Nield-Kuznetsov functions (\ref{03}) via
\begin{equation} \label{63}
L_{W}^{\pm}(a,\lambda)  =\pm W'(a,0) W_{0}^{(0,3)}(a,2\lambda)
+\tfrac{1}{2}W(a,0) W_{1}^{(0,3)}(a,2\lambda).
\end{equation}
\end{theorem}

\begin{proof}
The differential equation (\ref{60a}) is an inhomogeneous form of (\ref{02a}) with $z$ replaced $2\lambda$. Thus
\begin{multline} \label{63a}
L_{W}^{\pm}(a,\lambda)  =\pm W'(a,0) W_{0}^{(0,3)}(a,2\lambda)
+\tfrac{1}{2}W(a,0) W_{1}^{(0,3)}(a,2\lambda) \\
+A_{0}(a)W_{0}(a,2\lambda)+A_{3}(a)W_{3}(a,2\lambda),
\end{multline}
for some constants $A_{0}(a)$ and $A_{3}(a)$. Now as $\lambda \rightarrow \infty$ with $\delta\leq \arg(\lambda) \leq \frac{1}{2}\pi-\delta$ all terms on the RHS of (\ref{63a}) vanish, with the exception of $W_{3}(a,2\lambda)$ which is unbounded. But from (\ref{60d}) the LHS vanishes in this limit, and hence $A_{3}(a)=0$. Similarly by letting $\lambda \rightarrow \infty$ with $-\frac{1}{2}\pi+\delta \leq\arg(\lambda) \leq -\delta<0$ we deduce that $A_{0}(a)=0$.
\end{proof}

We now proceed to obtain uniform asymptotic expansions using these representations. If we let $u=2a$ and define $z$ by
\begin{equation} \label{lambda}
2\lambda = \sqrt{2u} \, z,
\end{equation}
then from using the asymptotic expansions of \cref{thm3} in the RHS of (\ref{63}) we have under the conditions of that theorem
\begin{multline} \label{90}
L_{W}^{\pm}\left(\tfrac{1}{2}u,\sqrt{\tfrac{1}{2}u}\,z\right)  \sim 
\pm \frac{2}{u}W'\left(\tfrac{1}{2}u,0\right)
\sum\limits_{s=0}^{\infty} \frac{G_{s,0}(z)}{u^{2s}} 
+\sqrt{\frac{2}{u}}W\left(\tfrac{1}{2}u,0\right)
\sum\limits_{s=0}^{\infty} \frac{G_{s,1}(z)}{u^{2s}}
\\ 
+\nu^{\pm}(u)\left\{\pi \mathrm{Hi}\left(u^{2/3}\zeta  \right)\mathcal{A}(u,z)  +  \pi \mathrm{Hi}^{\prime}\left(u^{2/3}\zeta \right)\mathcal{B}(u,z)
-\frac{J(u,z)}{u^{2/3}\zeta }\left(\frac{\zeta}{1-z^2}\right)^{1/4}\right\},
\end{multline}
where
\begin{equation} \label{nu}
\nu^{\pm}(u)=
\pm 2u\gamma_{0}(u)W'\left(\tfrac{1}{2}u,0\right)
+\sqrt{2}u^{3/2}\gamma_{1}(u)W\left(\tfrac{1}{2}u,0\right).
\end{equation}

Consider first $L_{W}^{+}(\frac{1}{2}u,\sqrt{\frac{1}{2}u}\,z)$. We find from (\ref{15}), (\ref{16}), (\ref{stirling}), (\ref{34}), (\ref{34.1}), and (\ref{92}) for large $u$ and arbitrary positive integer $n$ that
\begin{equation}  \label{93}
2u\gamma_{0}(u)W'\left(\tfrac{1}{2}u,0\right)
=-2^{-1/4}u^{-1/12}
\left\{1+\mathcal{O}\left(u^{-n}\right)\right\},
\end{equation}
and
\begin{equation}  \label{94}
\sqrt{2}u^{3/2}\gamma_{1}(u)W\left(\tfrac{1}{2}u,0\right)
=2^{-1/4}u^{-1/12}
\left\{1+\mathcal{O}\left(u^{-n}\right)\right\}.
\end{equation}
 Now from these two equations and (\ref{nu}) we see that $\nu^{+}(u)=\mathcal{O}(u^{-n})$, and hence from (\ref{15}), (\ref{16}), (\ref{stirling}), (\ref{39}), (\ref{40}), and (\ref{90}) we find that
\begin{equation} \label{95}
L_{W}^{+}\left(\tfrac{1}{2}u,\sqrt{\tfrac{1}{2}u}\,z\right)  \sim 
\frac{2^{1/4}}{u^{3/4}}\sum\limits_{s=0}^{\infty} \frac{G_{s}^{+}(z)}{u^{2s}},
\end{equation}
where $G_{s}^{+}(z)$ are coefficients in the formal expansion
\begin{equation} \label{105c}
\frac{2}{u}W'\left(\tfrac{1}{2}u,0\right)
\sum\limits_{s=0}^{\infty} \frac{G_{s,0}(z)}{u^{2s}} 
+\sqrt{\frac{2}{u}}W\left(\tfrac{1}{2}u,0\right)
\sum\limits_{s=0}^{\infty} \frac{G_{s,1}(z)}{u^{2s}}
\sim \frac{2^{1/4}}{u^{3/4}}\sum\limits_{s=0}^{\infty} 
\frac{G_{s}^{+}(z)}{u^{2s}}.
\end{equation}
These can be computed by using (\ref{15}), (\ref{16}), (\ref{stirling}), (\ref{39}), and (\ref{40}). For the first two we find
\begin{equation} \label{95a}
G_{0}^{+}(z)=\frac{1}{z+1},
\end{equation}
and
\begin{equation} \label{95b}
G_{1}^{+}(z)=\frac { (z+3) \left( {z}^{2}+2z+5 \right) }{8(z+1)^{4}}.
\end{equation}

The expansion (\ref{95}) is certainly valid to the right of the curves $\Re(\xi)=0$ shown in \cref{fig:fig1}, but it is not clear if the term in braces in (\ref{90}) is negligible elsewhere, or indeed if the all the coefficients are defined at $z=1$. To verify this, and improve accuracy, we obtain an improved version of the expansion. We do this by using the connection formula \cite[Eqs. (5.26) and (5.27)]{Dunster:2021:UAP}
\begin{equation} \label{96}
W_{R}^{(0,3)}\left(\tfrac{1}{2}u,\sqrt{2u} \, z\right)
=W_{R}^{(0,2)}\left(\tfrac{1}{2}u,\sqrt{2u} \, z\right)
+\mu_{R}(u) W_{0}\left(\tfrac{1}{2}u,\sqrt{2u} \, z\right),
\end{equation}
where
\begin{equation} \label{97}
\mu_{0}(u)=2^{-\frac{1}{4}+\frac{1}{4}iu} e^{\pi i/4} e^{u\pi/4}
\Gamma\left( \tfrac{1}{4}+\tfrac {1}{4}iu \right),
\end{equation}
and
\begin{equation} \label{98}
\mu_{1}(u)=
2^{\frac{5}{4}+\frac{1}{4}iu} e^{u\pi/4}
\Gamma\left( \tfrac{3}{4}+\tfrac {1}{4}iu \right).
\end{equation}
Thus from (\ref{63}) and (\ref{96}) we have the exact expression
\begin{multline} \label{99}
L_{W}^{+}\left(\tfrac{1}{2}u,\sqrt{\tfrac{1}{2}u}z\right)  
= W'\left(\tfrac{1}{2}u,0\right) W_{0}^{(0,2)}\left(\tfrac{1}{2}u,\sqrt{2u} \, z\right)
\\
+\tfrac{1}{2}W\left(\tfrac{1}{2}u,0\right) W_{1}^{(0,2)}\left(\tfrac{1}{2}u,\sqrt{2u} \, z\right)
+\mu(u) W_{0}\left(\tfrac{1}{2}u,\sqrt{2u} \, z\right),
\end{multline}
where
\begin{equation} \label{mu}
\mu(u)=
\mu_{0}(u)W'\left(\tfrac{1}{2}u,0\right)
+\tfrac{1}{2}\mu_{1}(u)W\left(\tfrac{1}{2}u,0\right).
\end{equation}

Similarly, in terms of $W_{0}^{(1,3)}(\frac{1}{2}u,\sqrt{2u} \, z)$ instead of $W_{0}^{(0,2)}(\frac{1}{2}u,\sqrt{2u} \, z)$, we can show that
\begin{multline} \label{99a}
L_{W}^{+}\left(\tfrac{1}{2}u,\sqrt{\tfrac{1}{2}u}z\right)  
= W'\left(\tfrac{1}{2}u,0\right) W_{0}^{(1,3)}\left(\tfrac{1}{2}u,\sqrt{2u} \, z\right)
\\
+\tfrac{1}{2}W\left(\tfrac{1}{2}u,0\right) W_{1}^{(1,3)}\left(\tfrac{1}{2}u,\sqrt{2u} \, z\right)
+\overline{\mu(u)} W_{3}\left(\tfrac{1}{2}u,\sqrt{2u} \, z\right).
\end{multline}

At this stage it is worth emphasizing for large $z$ and either $R=0$ or $R=1$ that $W_{R}^{(0,2)}(\frac{1}{2}u,\sqrt{2u} \, z)$ is characterized as being bounded in the first and third quadrants, and $W_{R}^{(1,3)}(\frac{1}{2}u,\sqrt{2u} \, z)$ is characterized as being bounded in the second and fourth quadrants. While (\ref{99}) and (\ref{99a}) can both in theory be used for all $z$, we only use them in the quadrants in which these functions are bounded. Taken together, both of them cover the whole complex plane.

Thus for example, for the right half-plane (\ref{99}) (respectively (\ref{99a})) is only numerically useful in the first (respectively fourth) quadrant, since in the other quadrant all functions of the RHS of (\ref{99}) and (\ref{99a}) are exponentially large, with large cancellations due to the fact $L_{W}^{+}(\tfrac{1}{2}u,\sqrt{\frac{1}{2}u}z)$ is bounded in the right half-plane to the right of the curves depicted in \cref{fig:fig1} emanating from $z=1$. 

In summary, since $W_{R}^{(0,2)}(\tfrac{1}{2}u,\sqrt{2u} \, z)$ is bounded in the first and third quadrants we use (\ref{99}) in these quadrants (excluding points on or near the cut $-\infty < z \leq -1$ where the asymptotic expansions break down), and likewise we use (\ref{99a}) in the conjugate region.

The asymptotic expansion for $W_{R}^{(0,2)}(\tfrac{1}{2}u,\sqrt{2u} \, z)$ is the same as (\ref{34b}) except with the Scorer function $\mathrm{Hi}(z)$ replaced by its rotated form $e^{2 \pi i/3}\mathrm{Hi}(ze^{2\pi i/3})$. We thus find similarly to the derivation of the expansion (\ref{95}) that
\begin{multline} \label{99b}
W'\left(\tfrac{1}{2}u,0\right) W_{0}^{(0,2)}\left(\tfrac{1}{2}u,\sqrt{2u} \, z\right)
\\
+\tfrac{1}{2}W\left(\tfrac{1}{2}u,0\right) W_{1}^{(0,2)}\left(\tfrac{1}{2}u,\sqrt{2u} \, z\right)
\sim \frac{2^{1/4}}{u^{3/4}}\sum\limits_{s=0}^{\infty} \frac{G_{s}^{+}(z)}{u^{2s}}.
\end{multline}
For $0\leq \arg(\pm z) \leq \pi/2$ and $u \rightarrow \infty$ we then have from (\ref{99}) and (\ref{99b})
\begin{equation} \label{103}
L_{W}^{+}\left(\tfrac{1}{2}u,\sqrt{\tfrac{1}{2}u}\,z\right)  \sim 
\frac{2^{1/4}}{u^{3/4}}\sum\limits_{s=0}^{\infty} \frac{G_{s}^{+}(z)}{u^{2s}} 
+\mu(u) W_{0}\left(\tfrac{1}{2}u,\sqrt{2u} \, z\right),
\end{equation}
and similarly from (\ref{99a}) for $-\pi /2 \leq \arg(\pm z) \leq 0$
\begin{equation} \label{104}
L_{W}^{+}\left(\tfrac{1}{2}u,\sqrt{\tfrac{1}{2}u}\,z\right)  \sim 
\frac{2^{1/4}}{u^{3/4}}\sum\limits_{s=0}^{\infty} \frac{G_{s}^{+}(z)}{u^{2s}} 
+\overline{\mu(u)} W_{3}\left(\tfrac{1}{2}u,\sqrt{2u} \, z\right),
\end{equation}
in both cases excluding points on or near the cut $-\infty < z \leq -1$.

Finally, after some calculation using (\ref{15}), (\ref{16}), (\ref{reflect}), (\ref{97}), (\ref{98}), and (\ref{mu}) it can be shown that
\begin{equation} \label{100}
\mu(u)=
\left(\frac{\pi e^{\pi u}}{2}\right)^{1/4}
\left\{
\Gamma\left( \tfrac{1}{2}+\tfrac {1}{2}iu \right)
\right\}^{1/2}
\left[T^{1/4}(u)-e^{\pi i /4}T^{-1/4}(u)
\right],
\end{equation}
where
\begin{equation} \label{101}
T(u)=\frac{1 + i \tanh\left(\tfrac{1}{4} \pi u \right)}
{1 - i \tanh\left(\tfrac{1}{4} \pi u \right)}.
\end{equation}
From this it follows that
\begin{equation} \label{102}
\mu(u)=e^{5\pi i /8}
\left(\frac{\pi }{2 e^{\pi u}}\right)^{1/4}
\left\{\Gamma\left( \tfrac{1}{2}+\tfrac {1}{2}iu \right)
\right\}^{1/2}
\left\{1+\mathcal{O}\left(e^{-\pi u}\right)\right\},
\end{equation}
which incidentally (with the aid of (\ref{stirling})) shows that $\mu(u)=\mathcal{O}(e^{-3\pi u/8})$.

Our first desired uniform asymptotic expansion is then given as follows.
\begin{theorem} \label{thm4.4}
Let $\mathrm{Ai}_{\mp 1}(z)=\mathrm{Ai}(z e^{\pm 2\pi i /3})$. Then as $u \rightarrow \infty$ we have uniformly
\begin{multline} \label{104a}
L_{W}^{+}\left(\tfrac{1}{2}u,\sqrt{\tfrac{1}{2}u}\,z\right)  \sim 
\frac{2^{1/4}}{u^{3/4}}\sum\limits_{s=0}^{\infty} 
\frac{G_{s}^{+}(z)}{u^{2s}} 
+\sqrt{2}\pi^{3/4} u^{-1/12}
e^{-\frac{3}{8}\pi u \pm \frac {2}{3} \pi i}
{\left\vert
\Gamma\left( \tfrac{1}{2}+\tfrac {1}{2}iu \right)
\right\vert} ^{1/2} \\
\times \left\{
\mathrm{Ai}_{\mp 1}\left(u^{2/3}\zeta \right) 
\mathcal{A}(u,z) 
+\mathrm{Ai}_{\mp 1}^{\prime }\left(u^{2/3}\zeta\right) \mathcal{B}(u,z)\right\},
\end{multline}
where upper signs are taken for $z$ in the first and third quadrants, and lower signs for $z$ in the second and fourth quadrants. In both cases points within a distance of $\delta$ of the cut $-\infty < z \leq -1$ must be excluded, but otherwise $z$ is unrestricted. The coefficients $G_{s}^{+}(z)$ are rational functions whose only poles are at $z=-1$, and are given by (\ref{105c}) using (\ref{15}), (\ref{16}), (\ref{stirling}), (\ref{39}), (\ref{40}). Also, $\zeta$ is given by (\ref{zeta1}), and $\mathcal{A}(u,z)$ and $\mathcal{B}(u,z)$ are functions (analytic at $z=1$) having the expansions (\ref{31}) and (\ref{32}), respectively.
\end{theorem}

\begin{proof}
The expansions (\ref{104a}) follow from (\ref{30}), (\ref{41a}), (\ref{103}), (\ref{104}), and (\ref{102}). Next one can show for $|\zeta|<\rho$, where $\rho>0$ is sufficiently small, that the second term on the RHS of (\ref{104a}), involving the Airy function and its derivative, is exponentially small for large $u$. Now $\zeta=0$ corresponds to $z=1$, and therefore (\ref{95}) holds for $0<\rho_{1}<|z-1|<\rho_{2}$ for certain constants $\rho_{1}$ and $\rho_{2}$. We then apply \cref{thm:nopoles} to establish the analyticity of $G_{s}^{+}(z)$ at $z=1$, as well as the validity of (\ref{104a}) in a neighborhood of $z=1$.
\end{proof}

Let us now turn our attention to $L_{W}^{-}(\frac{1}{2}u,\sqrt{\frac{1}{2}u}\,z)$. Then, similarly to (\ref{105c}) we define coefficients $G_{s}^{-}(z)$ via the formal expansion
\begin{multline} \label{105}
- \frac{2}{u}W'\left(\tfrac{1}{2}u,0\right)
\sum\limits_{s=0}^{\infty} \frac{G_{s,0}(z)}{u^{2s}} 
+\sqrt{\frac{2}{u}}W\left(\tfrac{1}{2}u,0\right)
\sum\limits_{s=0}^{\infty} \frac{G_{s,1}(z)}{u^{2s}}
\\ 
- \frac{ J(u,z)}{u^{3/4}\zeta }\left(\frac{8\zeta}{1-z^2}\right)^{1/4}
 \sim 
\frac{2^{1/4}}{u^{3/4}}\sum\limits_{s=0}^{\infty} 
\frac{G_{s}^{-}(z)}{u^{2s}},
\end{multline}
where $J(u,z)$ is a function (analytic at $z=1$) having the asymptotic expansion (\ref{36}).

The first two coefficients are found to be
\begin{equation} \label{107}
G_{0}^{-}(z)=\frac{1}{z-1}
+\frac {\sqrt{2}\phi(z) }{\zeta},
\end{equation}
and
\begin{multline} \label{108}
G_{1}^{-}(z)=\frac{(z-3)\left( z^{2}-2z+5 \right) }{8 (z-1)^{4}}
+\frac {\phi(z) \left\{ \tilde{\mathcal{E}}_{1}^{2}(z)
+2 \tilde{\mathcal{E}}_{2}(z)\right\} }{\sqrt {2}\zeta} \\
-\frac{\sqrt {2}\phi(z)
\mathcal{E}_{1}(z)}{\zeta^{5/2}}
+\frac{2^{3/2}\phi(z)}{\zeta^{4}}.
\end{multline}
Unlike $G_{s}^{+}(z)$ these are not rational functions of $z$, however from \cref{thm:nopoles} one can show that they also have removable singularities at $z=1$ ($\zeta=0$), and so can be considered analytic at $z=1$. For example, using Maple we find as $z \rightarrow 1$ for the first two
\begin{equation} \label{110}
G_{0}^{-}(z)=\frac{1}{5}
-\frac {89}{1400}(z-1)
+\mathcal{O}\left\{(z-1)^2\right\},
\end{equation}
and
\begin{equation} \label{111}
G_{1}^{-}(z)=\frac{586349}{8624000}
-\frac {221592597}{5605600000}(z-1)
+\mathcal{O}\left\{(z-1)^2\right\}.
\end{equation}

From (\ref{90}), (\ref{nu}), (\ref{93}), (\ref{94}), and (\ref{105}) we now obtain our desired result.
\begin{theorem} \label{thm4.5}
As $u \rightarrow \infty$
\begin{multline} \label{106}
L_{W}^{-}\left(\tfrac{1}{2}u,\sqrt{\tfrac{1}{2}u}\,z\right)  \sim 
\frac{2^{1/4}}{u^{3/4}}\sum\limits_{s=0}^{\infty} 
\frac{G^{-}_{s}(z)}{u^{2s}}
\\ 
+\frac{2^{3/4}\pi}{u^{1/12}} 
\left\{ \mathrm{Hi}\left(u^{2/3}\zeta  \right)\mathcal{A}(u,z)  +  \mathrm{Hi}^{\prime}\left(u^{2/3}\zeta \right)\mathcal{B}(u,z)
\right\},
\end{multline}
uniformly for all $z$ except those within a distance of $\delta$ of the cut $-\infty < z \leq -1$. Here $\mathrm{Hi}(z)$ is the Scorer function defined by (\ref{eqHi}).
\end{theorem}

The following theorem provides asymptotic expansions uniformly valid in a region containing the interval $-\infty < z \leq -1$, which was excluded from \cref{thm4.4,thm4.5}.
\begin{theorem} \label{thm4.6}
Let $z$ lie on, or to the right of the level curves $\Re(\xi)=0$ in the first and fourth quadrants, emanating from $z=1$, as depicted in \cref{fig:fig1}, where $\xi$ is given by (\ref{zeta1}). Then as $u \rightarrow \infty$ we have uniformly for $z$ in this unbounded closed region
\begin{multline} \label{120}
L_{W}^{+}\left(\tfrac{1}{2}u,-\sqrt{\tfrac{1}{2}u}\,z\right) \sim -L_{W}^{-}\left(\tfrac{1}{2}u,\sqrt{\tfrac{1}{2}u}\,z\right) \\
+\sqrt{2}\pi^{3/4} u^{-1/12}
e^{\pi u/8}
{\left\vert
\Gamma\left( \tfrac{1}{2}+\tfrac {1}{2}iu \right)
\right\vert} ^{1/2} \\
\times
\left\{\mathrm{Bi}\left(u^{2/3}\zeta \right) 
\mathcal{A}(u,z) +\mathrm{Bi}^{\prime }\left(u^{2/3}\zeta\right) \mathcal{B}(u,z)\right\},
\end{multline}
and
\begin{multline} \label{122}
L_{W}^{-}\left(\tfrac{1}{2}u,-\sqrt{\tfrac{1}{2}u}\,z\right) \sim -L_{W}^{+}\left(\tfrac{1}{2}u,\sqrt{\tfrac{1}{2}u}\,z\right) \\
+2\sqrt{2}\pi^{3/4} u^{-1/12}
e^{5\pi u/8}
{\left\vert
\Gamma\left( \tfrac{1}{2}+\tfrac {1}{2}iu \right)
\right\vert} ^{1/2} \\
\times
\left\{\mathrm{Ai}\left(u^{2/3}\zeta \right) 
\mathcal{A}(u,z) +\mathrm{Ai}^{\prime }\left(u^{2/3}\zeta\right) \mathcal{B}(u,z)\right\}.
\end{multline}
\end{theorem}

\begin{remark}
In these the expansions of $L_{W}^{\pm}(\tfrac{1}{2}u,\sqrt{\tfrac{1}{2}u}\,z)$ given by \cref{thm4.4,thm4.5} can be used. Also, the Airy functions in both these expansions are exponentially large as $u \rightarrow \infty$ in the interior of the region of validity, except on the real axis where they are oscillatory for $1 \leq z < \infty$.
\end{remark}

\begin{proof}
From \cite[Eq. (5.33)]{Dunster:2021:UAP} we have for $R=0,1$
\begin{multline} \label{112}
W_{R}^{(0,3)}\left(\tfrac{1}{2}u,-\sqrt {2u} z\right)
=(-1)^{R}W_{R}^{(0,3)}\left(\tfrac{1}{2}u,\sqrt {2u} z\right)  \\
+\left\{(-1)^{R+1}-ie^{\pi u/2}
\right\}\mu_{R}(u)
W_{0}\left(\tfrac{1}{2}u,\sqrt {2u} z\right)  \\
+\left\{(-1)^{R+1}+ie^{\pi u/2}
\right\}\overline{\mu_{R}(u)}W_{3}\left(\tfrac{1}{2}u,\sqrt {2u} z\right),
\end{multline}
where $\mu_{0}(u)$ and $\mu_{1}(u)$ are given by (\ref{97}) and (\ref{98}), respectively. Thus inserting (\ref{112}) into (\ref{63}) and using (\ref{mu}) we obtain
\begin{multline} \label{113}
L_{W}^{\pm}\left(\tfrac{1}{2}u,-\sqrt{\tfrac{1}{2}u}\,z\right)  =-L_{W}^{\mp}\left(\tfrac{1}{2}u,\sqrt{\tfrac{1}{2}u}\,z\right)
+\hat{\mu}^{\pm}(u) W_{0}\left(\tfrac{1}{2}u,\sqrt {2u} z\right) \\
+ \overline{\hat{\mu}^{\pm}(u)} 
W_{3}\left(\tfrac{1}{2}u,\sqrt {2u} z\right),
\end{multline}
where
\begin{equation} \label{115}
\hat{\mu}^{+}(u)=
\tfrac{1}{2}\mu_{1}(u)W\left(\tfrac{1}{2}u,0\right)
-\mu_{0}(u)W'\left(\tfrac{1}{2}u,0\right)
-i e^{\pi u /2} \mu(u),
\end{equation}
and
\begin{equation} \label{116}
\hat{\mu}^{-}(u)=\mu(u)
- i e^{\pi u /2} \left\{
\tfrac{1}{2}\mu_{1}(u)W\left(\tfrac{1}{2}u,0\right)
-\mu_{0}(u)W'\left(\tfrac{1}{2}u,0\right) \right\}.
\end{equation}
We find from (\ref{15}), (\ref{16}), (\ref{reflect}), (\ref{97}), (\ref{98}), and (\ref{100}), for large $u$,
\begin{equation} \label{117}
\mu^{+}(u)=e^{\pi i /8}
\left(\frac{\pi e^{\pi u}}{2}\right)^{1/4}
\left\{\Gamma\left( \tfrac{1}{2}+\tfrac {1}{2}iu \right)
\right\}^{1/2}
\left\{1+\mathcal{O}\left(e^{-\pi u}\right)\right\},
\end{equation}
and
\begin{equation} \label{118}
\mu^{-}(u)=2^{3/4} \pi^{1/4} e^{-3 \pi i /8}  e^{3\pi u/4}
\left\{\Gamma\left( \tfrac{1}{2}+\tfrac {1}{2}iu \right)
\right\}^{1/2}
\left\{1+\mathcal{O}\left(e^{-\pi u}\right)\right\}.
\end{equation}

From (\ref{stirling}), (\ref{117}), and (\ref{118}) it can be shown that $\mu^{+}(u)=\mathcal{O}(e^{\pi u/8})$ and $\mu^{-}(u)=\mathcal{O}(e^{5\pi u/8})$ as $u \rightarrow \infty$. Now from \cite[Eq. 9.2.10]{NIST:DLMF}
\begin{equation} \label{119}
\mathrm{Bi}(z)=e^{-\pi i /6}\mathrm{Ai}_{1}(z)
+e^{\pi i/6}\mathrm{Ai}_{-1}(z),
\end{equation}
and so from this and (\ref{92}), (\ref{30}), (\ref{41a}), (\ref{113}), (\ref{115}), and (\ref{117}) we obtain (\ref{120}), provided $z$ lies in the stated closed region. To the left of the boundary curves this expansion cannot be used, since for large $u$ there is a severe cancellation of exponentially large terms on the RHS, whereas the LHS is bounded. However, this region is covered by \cref{thm4.4}.

Similarly, using \cite[Eq. 9.2.12]{NIST:DLMF}
\begin{equation} \label{121}
\mathrm{Ai}(z)=e^{\pi i /3}\mathrm{Ai}_{1}(z)
+e^{-\pi i/3}\mathrm{Ai}_{-1}(z),
\end{equation}
along with (\ref{92}), (\ref{30}), (\ref{41a}), (\ref{113}), (\ref{116}), and (\ref{118}) we obtain (\ref{122}). Unlike (\ref{120}) this is actually still valid in part of the region to the left of the boundary, but again we can use \cref{thm4.5} in this case.
\end{proof}

\section{Laplace Transform of \texorpdfstring{$U(a,t)$}{Lg}} 
\label{sec5}
Let us now consider
\begin{equation} \label{64}
L_{U}(a,\lambda)  =\int_{0}^{\infty} e^{-\lambda t} U(a,t)dt,
\end{equation}
which, like $L_{W}^{\pm}(a,\lambda)$, is an entire function of $\lambda$.

The procedures here are similar to those of \cref{sec4}, although a major difference is that we obtain a simpler slowly varying expansion for this Laplace transform that is valid in the whole right half $\lambda$ plane, and also in part of the left half-plane. To avoid introducing new terms we shall use the same notation in this section as the previous one, noting any differences of the functions and parameters involved.

Beginning with large $\lambda$ with $a$ fixed, similarly to \cref{thm4.1} we obtain an asymptotic expansion for $L_{U}(a,\lambda)$. The main difference here is that, unlike $W(a,\pm t)$, $U(a,t)$ is exponentially small in a sector containing the real $t$ axis: see (\ref{64a}). This allows us to apply a more general version of Watson's lemma for analytic integrands which are exponentially small in sectors containing the positive real $t$ axis \cite[Chap. 4, Thm. 3.3]{Olver:1997:ASF}. So from this and \cite[Eqs. 12.2.6 and 12.2.7]{NIST:DLMF} the following is obtained.

\begin{theorem} \label{thm5.1}
For fixed $a$ and $\lambda \rightarrow \infty$ with $|\arg(\lambda)| \leq \frac{3}{4}\pi - \delta$
\begin{equation} \label{67a}
L_{U}(a,\lambda) \sim
\frac{U(a,0)}{\lambda}
\sum_{ s=0}^{\infty} \frac{\alpha_{s}(a)}{\lambda^{2s}}
+\frac{U'(a,0)}{\lambda^{2}}
\sum_{ s=0}^{\infty} \frac{\beta_{s}(a)}{\lambda^{2s}},
\end{equation}
where $\alpha_{0}(a)=\beta_{0}(a)=1$, $\alpha_{1}(a)=\beta_{1}(a)=a$, and for $ s=0,1,2,\ldots$ 
\begin{equation}\label{67b}
\alpha_{s+2}=a\alpha_{s+1}+\tfrac{1}{2}(s+1)(2s+1)\alpha_{s},
\end{equation}
\begin{equation} \label{67c}
\beta_{s+2}=a\beta_{s+1}+\tfrac{1}{2}(s+1)(2s+3)\beta_{s},
\end{equation}
\begin{equation} \label{66}
U(a,0)=\frac{\sqrt{\pi}}{2^{\frac{1}{2}a
+\frac{1}{4}}\Gamma\left(\frac{1}{2}a+\frac{3}{4}\right)},
\end{equation}
and
\begin{equation} \label{67}
U'(a,0)=-\frac{\sqrt{\pi}}{2^{\frac{1}{2}a
-\frac{1}{4}}\Gamma\left(\frac{1}{2}a+\frac{1}{4}\right)}.
\end{equation}
\end{theorem}

We omit the proof since it is very similar to that of \cref{thm4.1}, although we note that the $\lambda$ sector of asymptotic validity in \cref{thm5.1} is larger.

Next we express $L_{U}(a,\lambda)$ in terms of certain functions $U_{R}^{(j,k)}(a,z)$ ($(j,k)=(0,1)$, $(0,3)$ and $(1,3)$). These are solutions of the inhomogeneous version of (\ref{UEQ}) with the RHS given by $z^R$ $(R=0,1)$. They have the fundamental properties at $z=\infty$ that $U_{R}^{(1,3)}(a,z)$ is bounded when $\frac{1}{4}\pi \leq |\arg(\pm z)|\leq\frac{3}{4}\pi$, $U_{R}^{(0,1)}(a,z)$ is bounded when $-\frac{1}{4}\pi \leq \arg(z)\leq\frac{3}{4}\pi$, and  $U_{R}^{(0,3)}(a,z)$ is bounded when $-\frac{3}{4}\pi \leq \arg(z)\leq\frac{1}{4}\pi$.

In \cite{Dunster:2021:UAP} they are defined similarly to (\ref{03}), namely
\begin{multline} \label{69}
 U_{R}^{(1,3)}(a,z)
=\frac{i\Gamma\left(\tfrac{1}{2}-a\right)}
{\sqrt{2\pi}} \left[U(-a,iz)
 \int_{i\infty}^z \right.
 t^{R}U(-a,-it)dt \\
\left. -U(-a,-iz) \int_{-i\infty}^z 
 t^{R}U(-a,it)dt \right],
\end{multline}
\begin{multline} \label{70s}
 U_{R}^{(0,1)}(a,z)
 =e^{(\frac{1}{2}a-\frac{1}{4})\pi i}
 \left[U(-a,-iz) \int_{\infty}^z \right.
 t^{R}U(a,t)dt \\
\left. -U(a,z) \int_{i\infty}^z 
 t^{R}U(-a,-it)dt  \right],
\end{multline}
and $U_{R}^{(0,3)}(a,z)$ be given by (\ref{70s}) with $i$ replaced by $-i$.

It turns out $L_{U}(a,\lambda)$ can be expressed in terms of any of these three functions (with $a$ replaced by $-a$), specifically we have the following.
\begin{theorem}
For $(j,k)=(0,1)$, $(0,3)$ or $(1,3)$
\begin{equation} \label{68}
L_{U}(a,\lambda)  =-U'(a,0) U_{0}^{(j,k)}(-a,2\lambda)
-\tfrac{1}{2}U(a,0) U_{1}^{(j,k)}(-a,2\lambda).
\end{equation}
\end{theorem}

\begin{proof}
Similarly to \cref{lem4.2} we find from (\ref{64})
\begin{equation} \label{65}
\frac{\partial^2  L_{U}(a,\lambda) }{\partial \lambda^2}
 =4\left(\lambda^2-a\right)L_{U}(a,\lambda)- 4U'(a,0)-4\lambda U(a,0).
\end{equation}
This is an inhomogeneous version of (\ref{UEQ}) with $z$ replaced by $2\lambda$ and $a$ replaced by $-a$. Now for $(j,k)=(0,1), (0,3)$ or $(1,3)$, and constants $A_{j}(a)$ and $A_{k}(a)$, the solution $L_{U}(a,\lambda)$ can be expressed in the form
\begin{multline} \label{68a}
L_{U}(a,\lambda)  =-U'(a,0) U_{0}^{(j,k)}(-a,2\lambda)
-\tfrac{1}{2}U(a,0) U_{1}^{(j,k)}(-a,2\lambda) \\
 +A_{j}(a)U_{j}(-a,2\lambda)+A_{k}(a)U_{k}(-a,2\lambda),
\end{multline}
where (analogously to (\ref{02b}))
\begin{equation} \label{68b}
U_{j}(a,z)
=U\left((-1)^{j}a,(-i)^{j}z\right) \quad (j=0,1,3),
\end{equation}
are solutions of (\ref{UEQ}) that are exponentially small at infinity in the sectors \linebreak $|\arg((-i)^{j}z)| <\frac{1}{4}\pi$ (see (\ref{64a})). Then by letting $\lambda \rightarrow \infty$ in various appropriate directions, and noting that from \cref{thm5.1} $L_{U}(a,\lambda)$ is bounded for $|\arg(\lambda)| \leq \frac{3}{4}\pi - \delta$, we deduce that $A_{j}(a)=A_{k}(a)=0$ in all three cases, and the result follows.
\end{proof}

\begin{figure}[hthp]
  \centering
  \includegraphics[trim={0 100 0 0},width=0.8\textwidth,keepaspectratio]{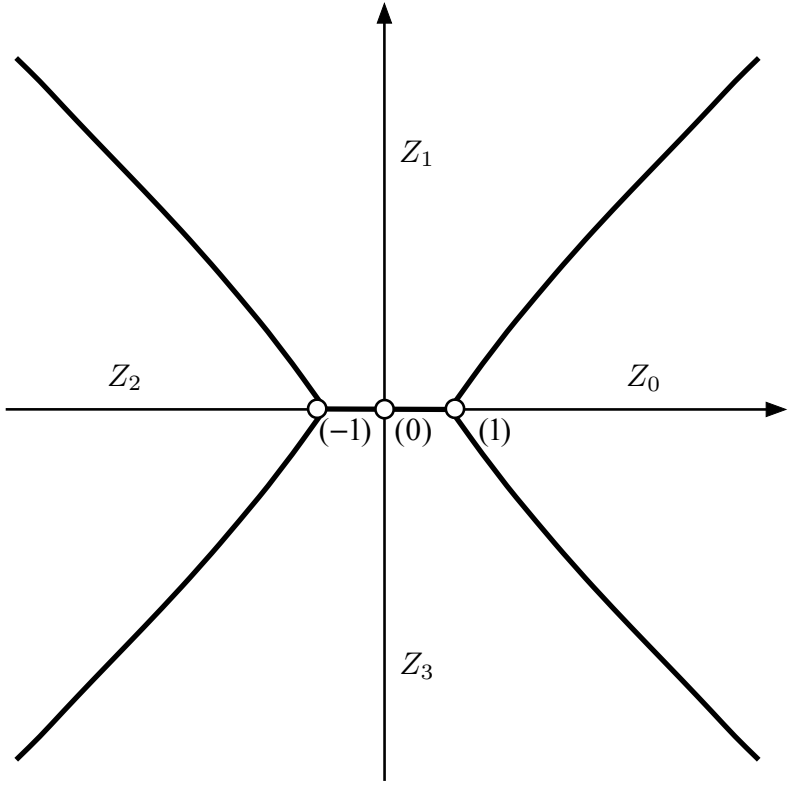}
  \caption{Domains $Z_{j}$ ($j=0,1,2,3$)}
  \label{fig:fig2}
\end{figure}

We again define $z$ by (\ref{lambda}), and with this introduce closed regions $Z_{j}$ ($j=0,1,2,3$) as depicted in \cref{fig:fig2}. In this the boundaries (thick lines) comprise the interval $-1 \leq z \leq 1$ and the curves emanating from $z= \pm 1$ given by $\Re(\xi)=\mathrm{constant}$, where instead of (\ref{zeta1}) we have
\begin{equation} \label{xxi}
\xi = \int_{1}^{z}\left(t^2-1\right)^{1/2} dt=\tfrac{1}{2}z\sqrt {z^{2}-1}-\tfrac{1}{2}\ln\left( z+\sqrt {z^{2}-1} \right),
\end{equation}
which is positive for $1<z<\infty$ and a continuous function of $z$ for $z \in \mathbb{C} \setminus (-\infty,1]$. Further define the (open) subset $Z^{(j,k)}$ to be the interior of $Z_{j} \cup Z_{k}$ ($j \neq k$).

Next, in place of (\ref{39}) and (\ref{40}), define coefficients
\begin{equation} \label{73}
G_{0,R}(z)=-z^{R}/\left(z^{2}-1\right),
\end{equation}
and
\begin{equation} \label{74}
G_{s+1,R}(z)=G''_{s,R}(z)/\left(z^{2}-1\right)
\quad (s=0,1,2,\ldots ).
\end{equation}
Then, with the aid of (\ref{stirling}), (\ref{66}), (\ref{67}), (\ref{73}), and (\ref{74}) we define a sequence of coefficients $\hat{G}_{s}(z)$ ($s=0,1,2,\ldots$) via the formal expansion (cf. (\ref{105c}))
\begin{multline} \label{209bb}
- \frac{2}{u}U'\left(\tfrac{1}{2}u,0\right)
\sum\limits_{s=0}^{\infty} \frac{G_{s,0}(z)}{u^{2s}} 
-\sqrt{\frac{2}{u}}U\left(\tfrac{1}{2}u,0\right)
\sum\limits_{s=0}^{\infty} \frac{G_{s,1}(z)}{u^{2s}}
\\ 
\sim \frac{2^{1/4}}{u^{3/4}}
\left(\frac{2e}{u}\right)^{u/4}
\sum\limits_{s=0}^{\infty} 
\frac{\hat{G}_{s}(z)}{u^{2s}}.
\end{multline}
The first three are found to be
\begin{equation} \label{209c}
\hat{G}_{0}(z)=\frac{1}{1+z},
\end{equation}
\begin{equation} \label{209e}
\hat{G}_{1}(z)=\frac{1}{24(1+z)},
\end{equation}
and
\begin{equation} \label{209f}
\hat{G}_{2}(z)=
-\frac{143z^{3}+717z^{2}+1581z+2159}{1152(1+z)^{4}}.
\end{equation}

Our main result then reads as follows.
\begin{theorem} \label{thm5.3}
The coefficients $\hat{G}_{s}(z)$ are rational functions whose only poles are at $z=-1$, and as $u \rightarrow \infty$
\begin{equation} \label{210f}
L_{U}\left(\tfrac{1}{2}u,\sqrt{\tfrac{1}{2}u}\,z\right)
\sim  \frac{2^{1/4}}{u^{3/4}}
\left(\frac{2e}{u}\right)^{u/4}
\sum\limits_{s=0}^{\infty} 
\frac{\hat{G}_{s}(z)}{u^{2s}},
\end{equation}
for $z$ in the interior of $Z_{0} \cup Z_{1} \cup Z_{3}$, uniformly for all points in this set whose distance from its boundary is greater than or equal to $\delta$.
\end{theorem}

\begin{proof}
It is evident from (\ref{73}) - (\ref{209bb}) that $\hat{G}_{s}(z)$ are rational functions, whose only possible poles are at $z=\pm 1$. Next, by applying \cite[Thm. 4]{Dunster:2020:ASI} we obtain the simple asymptotic expansions for large $u$
\begin{equation} \label{204}
U_{R}^{(j,k)}\left(-\tfrac{1}{2}u,\sqrt{2u} \, z\right)
\sim 
\frac{2(2u)^{\frac{1}{2}R}}{u}
\sum\limits_{s=0}^{\infty} \frac{G_{s,R}(z)}{u^{2s}} \quad (R=0,1),
\end{equation}
where $(j,k)=(0,1), (0,3)$ or $(1,3)$. These are valid for $z \in Z^{(j,k)}$, and uniformly for all points in this set whose distance from its boundary is greater than or equal to $\delta$. Note that $z= \pm 1$ are excluded from all three domains of validity, which is clear from (\ref{73}) and (\ref{74}).
Thus from (\ref{68}), (\ref{209bb}), and (\ref{204}) we see that (\ref{210f}) holds for all interior points of $Z_{0} \cup Z_{1} \cup Z_{3}$, except possibly in a closed neighborhood of $z=1$. However, from \cref{thm:nopoles} we see that the expansion is valid there too, and moreover the coefficients $\hat{G}_{s}(z)$ are analytic at $z=1$. Finally, the error bounds for (\ref{204}) supplied by \cite[Thm. 4]{Dunster:2020:ASI} establish that the expansion is uniformly valid for all (unbounded) $z$ that are at least a distance $\delta$ from the boundary of $Z_{0} \cup Z_{1} \cup Z_{3}$.
\end{proof}

It remains to provide an alternative expansion valid for a domain containing points excluded from \cref{thm5.3}. This is not quite as simple, since like \cref{thm4.5} it will involve the Scorer function.

We begin by defining, in place of (\ref{23}), (\ref{24}), (\ref{25}), and (\ref{26}), 
\begin{equation} \label{201}
\beta=\frac{z}{\sqrt{z^{2}-1}},
\end{equation}
where the branch of the square root is positive for $z>1$ and is continuous in the plane having a cut along $[-1,1]$. Thus $\beta \rightarrow 1$ as $z \rightarrow \infty$ in any direction. Further define
\begin{equation} \label{202}
\mathrm{E}_{1}(\beta)=\tfrac{1}{24}\beta
\left(5\beta^{2}-6\right),
\end{equation}
\begin{equation} \label{3.9}
\mathrm{E}_{2}(\beta)=
\tfrac{1}{16}\left(\beta^{2}-1\right)^{2} 
\left(5\beta^{2}-2\right),
\end{equation}
and for $s=2,3,4\ldots$
\begin{equation} \label{203}
\mathrm{E}_{s+1}(\beta) =
\frac{1}{2} \left(\beta^{2}-1 \right)^{2}\mathrm{E}_{s}^{\prime}(\beta)
+\frac{1}{2}\int_{\sigma(s)}^{\beta}
\left(p^{2}-1 \right)^{2}
\sum\limits_{j=1}^{s-1}
\mathrm{E}_{j}^{\prime}(p)
\mathrm{E}_{s-j}^{\prime}(p) dp,
\end{equation}
where again $\sigma(s)=1$ for $s$ odd and $\sigma(s)=0$ for $s$ even, so that the even and odd coefficients are even and odd functions of $\beta$, respectively, with $\mathrm{E}_{2s}(1)=0$ ($s=1,2,3,\ldots$).

Next, in place of (\ref{34}), (\ref{34.1}), (\ref{nu}) let
\begin{equation} \label{213}
\nu(u)=
- \sqrt{2}u\gamma_{0}(u)U'\left(\tfrac{1}{2}u,0\right)
+u^{3/2}\gamma_{1}(u)U\left(\tfrac{1}{2}u,0\right),
\end{equation}
where $\gamma_{R}(u)$ ($R=0,1$) are constants possessing the asymptotic expansions
\begin{equation} \label{203a}
\gamma_{R}(u) \sim
\frac{\sqrt {\pi} 
u^{\frac{1}{4}u-\frac{1}{2}R-\frac{13}{12}}}
{2^{\frac{1}{2}u-R+\frac{1}{2}}
\Gamma\left(\tfrac{1}{4}u-\tfrac{1}{2}R+\tfrac{3}{4}\right)}
\exp\left\{-\frac{1}{4}u
+\sum\limits_{s=0}^{\infty}\frac{\mathrm{E}_{2s+1}(1)}{u^{2s+1}}
\right\},
\end{equation}
as $u \rightarrow \infty$. Then let
\begin{equation} \label{211b}
\zeta =(3\xi/2)^{2/3},
\end{equation}
where $\xi$ is given by (\ref{xxi}) and the principal branch is taken. Also, similarly to (\ref{105}), coefficients $\tilde{G}_{s}(z)$ are given via the formal expansion
\begin{multline} \label{212}
- \frac{2}{u}U'\left(\tfrac{1}{2}u,0\right)
\sum\limits_{s=0}^{\infty} \frac{G_{s,0}(z)}{u^{2s}} 
+\sqrt{\frac{2}{u}}U\left(\tfrac{1}{2}u,0\right)
\sum\limits_{s=0}^{\infty} \frac{G_{s,1}(z)}{u^{2s}}
\\ 
-\frac{\sqrt{2} \nu(u)J(u,z)}{u^{2/3}\zeta }
\left(\frac{\zeta}{1-z^2}\right)^{1/4}
 \sim \frac{2^{1/4}}{u^{3/4}}
\left(\frac{2e}{u}\right)^{u/4}
\sum\limits_{s=0}^{\infty} 
\frac{\tilde{G}_{s}(z)}{u^{2s}},
\end{multline}
where $G_{s,R}(z)$ ($R=0,1$) are given by (\ref{73}) and (\ref{74}), and $J(u,z)$ has the expansion given by (\ref{36}) along with (\ref{27}), (\ref{28}), (\ref{29}), (\ref{201}) - (\ref{203}). These coefficients, like those of (\ref{105}), are analytic at $z=1$. The desired expansion is then given as follows.
\begin{theorem}
Let $\xi$ and $\zeta$ be given by (\ref{xxi}) and (\ref{211b}), and $\mathcal{A}(u,z)$ and $\mathcal{B}(u,z)$ have the expansions given by (\ref{27}) - (\ref{32}), where $\mathrm{E}_{s}(\beta)$ are defined by (\ref{201}) - (\ref{203}). Then as $u \rightarrow \infty$
\begin{multline} \label{211}
L_{U}\left(\tfrac{1}{2}u,-\sqrt{\tfrac{1}{2}u}\,z\right)  \sim 
\frac{2^{1/4}}{u^{3/4}}
\left(\frac{2e}{u}\right)^{u/4}
\Bigg[
\sum\limits_{s=0}^{\infty} 
\frac{\tilde{G}_{s}(z)}{u^{2s}}
\\ 
+\sqrt{2} \pi u^{2/3}\exp\left\{-\sum\limits_{s=0}^{\infty}
\frac{\mathrm{E}_{2s+1}(1)}{u^{2s+1}} \right\}
\left\{\mathrm{Hi}\left(u^{2/3}\zeta  \right)\mathcal{A}(u,z)  +  \mathrm{Hi}^{\prime}\left(u^{2/3}\zeta \right)\mathcal{B}(u,z)
\right\} \Bigg],
\end{multline}
for $z$ in the interior of $Z_{0} \cup Z_{1} \cup Z_{3}$, uniformly for all points in this set whose distance from its boundary is no less than $\delta$.
\end{theorem}

\begin{proof}
Firstly we note that $U_{R}^{(1,3)}(-a,-2\lambda)=(-1)^{R}U_{R}^{(1,3)}(-a,2\lambda)$ since both are particular solutions of the same differential equation with the same unique subdominant behavior at $\lambda = \pm i \infty$. Thus from (\ref{68}) with $(j,k)=(1,3)$ we have
\begin{multline} \label{210}
L_{U}(a,-\lambda)  =-U'(a,0) U_{0}^{(1,3)}(-a,-2\lambda)
-\tfrac{1}{2}U(a,0) U_{1}^{(1,3)}(-a,-2\lambda) \\
=-U'(a,0) U_{0}^{(1,3)}(-a,2\lambda)
+\tfrac{1}{2}U(a,0) U_{1}^{(1,3)}(-a,2\lambda).
\end{multline}

Now from \cite[Sect. 3.1]{Dunster:2021:UAP} we find for $R=0$ and $R=1$, and $z$ lying in the stated domain, that the functions $U_{R}^{(1,3)}(-\frac{1}{2}u,\sqrt{2u} \, z)$ possess the same expansions as $W_{R}^{(0,3)}(\frac{1}{2}u,\sqrt{2u} \, z)$ given by  \cref{thm3}, but with coefficients $G_{s,R}(z)$ and $\mathrm{E}_{s}(\beta)$ given by (\ref{73}), (\ref{74}), (\ref{201}) - (\ref{203}) and $\gamma_{R}(u)$ given by (\ref{203a}). We then insert these into (\ref{210}) (with $a=\frac{1}{2}u$ and $2\lambda =\sqrt{2u} \, z$) and use (\ref{212}), and this yields in the stated domain
\begin{multline} \label{211a}
L_{U}\left(\tfrac{1}{2}u,-\sqrt{\tfrac{1}{2}u}\,z\right)  \sim 
\frac{2^{1/4}}{u^{3/4}}
\left(\frac{2e}{u}\right)^{u/4}
\sum\limits_{s=0}^{\infty} 
\frac{\hat{G}_{s}(z)}{u^{2s}}
\\ 
+\sqrt{2} \pi \nu(u)\left\{\mathrm{Hi}\left(u^{2/3}\zeta  \right)\mathcal{A}(u,z)  +  \mathrm{Hi}^{\prime}\left(u^{2/3}\zeta \right)\mathcal{B}(u,z)  \right\}.
\end{multline}

Next from (\ref{stirling}), (\ref{66}), (\ref{67}), (\ref{213}), and (\ref{203a}) we find that for large $u$
\begin{equation} \label{214}
\nu(u)\sim
\frac{\sqrt {\pi}2^{3/4}}
{u^{1/12}
\Gamma\left(\frac{1}{2}u+\frac{1}{2} \right)}
\left(\frac{u}{2e}\right)^{u/4}
\exp\left\{\sum\limits_{s=0}^{\infty}
\frac{\mathrm{E}_{2s+1}(1)}{u^{2s+1}} \right\}.
\end{equation}
Now from \cite[Eqs. (3.20) and (3.24)]{Dunster:2021:UAP} we have for large $u$
\begin{equation} \label{216}
\frac{\sqrt {2\pi}}{
\Gamma\left(\frac{1}{2}u+\frac{1}{2} \right)}
\left(\frac{u}{2e}\right)^{u/2}
\sim 
\exp\left\{-2\sum\limits_{s=0}^{\infty}
\frac{\mathrm{E}_{2s+1}(1)}{u^{2s+1}} \right\},
\end{equation}
and thus from (\ref{214})
\begin{equation} \label{217}
\nu(u)\sim
\frac{2^{1/4}}{u^{1/12}}
\left(\frac{2e}{u}\right)^{u/4}
\exp\left\{-\sum\limits_{s=0}^{\infty}
\frac{\mathrm{E}_{2s+1}(1)}{u^{2s+1}} \right\}.
\end{equation}
Plugging this into (\ref{211a}) then yields (\ref{211}).
\end{proof}

\section*{Acknowledgments}
I thank the two anonymous referees for their detailed reports which helped improve the manuscript. Financial support from Ministerio de Ciencia e Innovaci\'on, Spain, project PGC2018-098279-B-I00 (MCIU/AEI/FEDER, UE) is acknowledged. 

\bibliographystyle{siamplain}
\bibliography{biblio}

\end{document}